\documentclass[dvips,generic,preprint]{imsart}

\RequirePackage[OT1]{fontenc}
\RequirePackage{amsthm,amsmath}
\RequirePackage[numbers]{natbib}
\RequirePackage[colorlinks,citecolor=blue,urlcolor=blue]{hyperref}
\RequirePackage{hypernat}

\startlocaldefs
\numberwithin{equation}{section} \theoremstyle{plain}

\endlocaldefs

\usepackage{enumerate}

\usepackage{amsmath}
\usepackage{amssymb}
\usepackage{color}
\usepackage{url}
\usepackage{graphicx}
\usepackage{hyperref}

\newcommand{\tr}{\mbox{\textnormal{tr\,}}}

\newcommand{\E}{\mbox{\textnormal{E}}}
\newcommand{\PP}{\mathbb{P}}
\newcommand{\R}{\mathbb{R}}
\newcommand{\Q}{\mathbb{Q}}

\newcommand{\N}{\mathbb{N}}

\newcommand{\half}{\mbox{\normalsize${\frac{1}{2}}$}}
\newcommand{\quart}{\mbox{\normalsize${\frac{1}{4}}$}}
\newcommand{\dd}{\textnormal{d}}

\newcommand{\wX}{\widetilde{X}}

\newcommand{\wa}{\widetilde{a}}

\newcommand{\wb}{\widetilde{b}}
\newcommand{\wB}{\widetilde{B}}

\newcommand{\wv}{\widetilde{v}}
\newcommand{\wu}{\widetilde{u}}
\newcommand{\wW}{\widetilde{W}}
\newcommand{\wZ}{\widetilde{Z}}
\newcommand{\wY}{\widetilde{Y}}

\newcommand{\wmu}{\widetilde{\mu}}

\newcommand{\wtheta}{\widetilde{\theta}}

\newcommand{\diag}{\textnormal{diag}}
\newcommand{\Id}{\textnormal{I}}

\newcommand{\Di}{\partial\mathcal{D}_i}

\newcommand{\XXi}{\partial\mathcal{X}_i}

\newcommand{\sA}{A}

\newcommand{\sgn}{\mbox{\textnormal{sgn\,}}}
\newcommand{\rank}{\textnormal{rank\,}}
\newcommand{\vect}{\textnormal{vec}}

\newcommand{\filpspace}{(\Omega,\mathcal{F},(\mathcal{F}_t),\PP)}

\newtheorem{theorem}{Theorem}[section]
\newtheorem{lemma}[theorem]{Lemma}

\newtheorem{prop}[theorem]{Proposition}
\newtheorem{cor}[theorem]{Corollary}

\theoremstyle{definition}
\newtheorem{definition}[theorem]{Definition}
\newtheorem{example}[theorem]{Example}

\theoremstyle{remark}
\newtheorem{remark}[theorem]{Remark}

\numberwithin{equation}{section}

\begin{document}

\begin{frontmatter}
\title{Affine diffusions with non-canonical state space}
\runtitle{Affine diffusions with non-canonical state space}

\begin{aug}
\author{\fnms{Peter} \snm{Spreij}\ead[label=e1]{spreij@uva.nl}}
and
\author{\fnms{Enno} \snm{Veerman}\ead[label=e2]{e.veerman@uva.nl}}

\runauthor{P. Spreij and E. Veerman}

\affiliation{University of Amsterdam}

\address{Korteweg-de Vries Institute for
Mathematics\\
Universiteit van Amsterdam \\
Science park 904\\
1098XH Amsterdam
\\The Netherlands\\
\printead{e1} \phantom{E-mail:\ }\printead*{e2}}

\end{aug}

\begin{abstract} Multidimensional affine diffusions have been studied in detail for the case of a canonical state
space.
We present results for general state spaces and provide a complete
characterization of all possible affine diffusions with polyhedral
and quadratic state space.  We give necessary and sufficient
conditions on the behavior of drift and diffusion on the boundary
of the state space in order to obtain invariance and to prove
strong existence and uniqueness.
\end{abstract}



\end{frontmatter}

\section{Introduction}
Affine diffusions, introduced in the pioneering paper~\cite{dk96}
by Duffie and Kan, are widely used in finance for modelling the
term structure of interest rates. Their main attraction lies in
the fact that they imply closed form expressions for  bond prices.
Affine diffusions are $p$-dimensional Markov processes that solve
an \emph{affine} stochastic differential equation (SDE) driven by
a Brownian motion, i.e.\ an SDE with a drift $\mu(x)$ and
diffusion matrix $\theta(x)$, both affine functions in the
argument $x$. There are three important issues in the theory of
affine diffusions, to wit
\begin{itemize}
\item
{\em stochastic invariance} of a subset $\mathcal{X}$ of $\R^p$,
the state space,
\item
the existence and uniqueness of strong solutions to the SDE with
values in $\mathcal{X}$,
\item
the validity of the so-called \emph{affine transform formula} for
exponential moments.
\end{itemize}

Most of the theory that has recently been developed, concerns
affine diffusions with a canonical state space
$\R^m_{\geq0}\times\R^{p-m}$, henceforth referred to as the state
space in {\em standard canonical form}, due to its tractable
appearance which might ease the verification of possible technical
conditions. The notion of a canonical state space has been
introduced in \cite{ds00}. Worth mentioning is the seminal paper
\cite{dfs03} by Duffie, Filipovi\'{c} and Schachermayer, who
provide a complete characterization of (regular) affine processes,
allowing jumps as well, under the assumption of a standard
canonical state space. Regarding the three issues mentioned above,
for affine SDEs with canonical state space it is relatively easy
to establish strong existence and uniqueness, as well as to derive
conditions for invariance, see e.g.~\cite{ds00,fm09}. Moreover,
until recently, the affine transform formula has only been fully
verified for affine diffusions with a standard canonical state
space, see \cite{fm09}.

The current paper together with a companion paper \cite{part1} contribute to the theory of affine diffusions with a
\emph{non-canonical state space}. We will characterize all affine diffusions, focussing our attention on \emph{polyhedrons} (of which the standard canonical state space is a special case) as well as \emph{quadratic}
state spaces (those of which the boundary is characterized by a
quadratic function), though more is possible. For example, the
matrix-valued affine processes and related Wishart processes
treated in \cite{cfmt09,wishart} have the cone of positive
semi-definite matrices as their state space. Our results extend
the classification of \cite{gs06} for the two-dimensional case to
higher dimensions. In \cite{gs06} it is shown that besides an
intersection of halfspaces, also a parabolic state space is
possible. We will see that in higher dimensions the quadratic
state spaces are not limited to the parabolic ones; there exist
also affine diffusions whose state space is a cone.

The present paper concerns the first two of three mentioned
issues for affine diffusions that live on a non-canonical
state space. Results on the third one are presented in the
companion paper~\cite{part1}, where we extend the results in
\cite{fm09} on the validity of the affine transform formula for
canonical to general state spaces. Returning to the first issue, in the current paper we derive conditions for the drift and diffusion matrix on the boundary of $\mathcal{X}$ to ensure stochastic invariance for both the polyhedral and the quadratic state space.
For the standard canonical state space these conditions are often called \emph{admissibility conditions},
see~\cite{ds00} and~\cite{dfs03}. The second issue, existence of a
unique strong solution to an affine SDE, is in general not
straightforward, as the square root of an affine matrix
valued function $\theta$ is not locally Lipschitz continuous for
singular $\theta$. This paper follows two approaches to solve this problem.

The first one is by invoking a result by Yamada and Watanabe
\cite[Theorem~1]{YamadaI}, as is done in \cite{ds00,fm09}. This
result is essentially only applicable for the standard canonical
state space. Under invariance conditions though, we prove that a
general polyhedral state space can be transformed in some kind of
canonical form, not necessarily the standard one, for which the
result by Yamada and Watanabe does apply. For a parabolic state
space unique strong solutions can be similarly obtained by
application of an appropriate modification of this result.

The second way to obtain unique strong solutions is to impose
conditions for invariance of $\{x\in\R^p:\theta(x)\mbox{ strictly positive definite}\}$ (also denoted by $\{\theta>0\})$, an approach followed
in \cite{dk96} for affine diffusions with a diagonalizable diffusion matrix and in \cite{pfaffel09} for matrix-valued diffusions. Strong existence and uniqueness is guaranteed, as the unique positive definite square root of $\theta$ is
locally Lipschitz continuous on $\{\theta>0\}$.  In the present
paper we derive invariance conditions for general state spaces,
following the arguments in \cite{pfaffel09}. This enables us to
obtain existence and uniqueness of affine diffusions whose state
space is a cone.

As a side note we mention that invariance of $\{\theta>0\}$ is
also important for applications. For example, in an affine term
structure model one often desires an affine structure of the
underlying SDE under both the risk-neutral and the physical
measure. For this purpose the invariance conditions for
$\{\theta>0\}$ are relevant in view of \cite[Corollary
A.9]{part1}, cf. \cite{duf02,cher2007}. Therefore, we will provide
these conditions not only for the cone but also for the polyhedral
and parabolic state space.

The remainder of this paper is organized as follows. After
introducing in Section~\ref{sec:def} some notation and defining
affine SDEs and diffusions more carefully, as well as presenting a
more detailed description of  our aims, we discuss in
Section~\ref{sec:stochinvar} stochastic invariance of the state
space. Necessary conditions (admissibility conditions) on the
behavior of the drift and the diffusion matrix on the boundary of
a general closed convex state space are derived, whereas
sufficient conditions are obtained for particular cases. In this
section we also provide sufficient conditions for stochastic
invariance of an open state space.

The invariance conditions derived in Section~\ref{sec:stochinvar}
are used in Section~\ref{sec:polyhedral} and
Section~\ref{sec:quadratic} to characterize all affine diffusions
with polyhedral respectively quadratic state space. For the former
we also give sufficient conditions, extending those known from the
literature (\cite{ds00,cfk08,fm09}), under which the diffusion
matrix can be diagonalized. In particular we show that the
classical model of \cite{dk96} can be transformed into the
canonical form of \cite{ds00}. The results from convex analysis
that we use in Section~\ref{sec:polyhedral} are stated and proved
in Appendix~\ref{sec:convexgeo}. In Section~\ref{sec:quadratic} we
show that for quadratic state spaces there are essentially only
two types of state spaces possible, a (multidimensional) parabola
and a cone. For each of these types we are able to give a full
characterization of the possible diffusion matrices.

\section{Definitions, approach and notation}\label{sec:def}

Let $p\in\N$. We are given a $p$-dimensional stochastic
differential equation
\begin{equation}\label{eq:SDE}
\dd X_t=\mu(X_t)\dd t+\sigma(X_t)\dd W_t,
\end{equation}
for continuous functions $\mu:\R^p\rightarrow\R^p$ and
$\sigma:\R^p\rightarrow\R^{p\times p}$ that satisfy the linear
growth condition
\begin{align}\label{eq:lingrow}
\|\mu(x)\|+\|\sigma(x)\|\leq M(1+\|x\|),\mbox{ for all $x\in\R^p$,
some $M>0$}.
\end{align}
By Theorems~IV.2.3 and IV.2.4 in \cite{Ikwat}, for all initial
conditions $x_0\in\R^p$ there exists a \emph{weak solution}
$(X,W)$ to (\ref{eq:SDE}), that is, there exists a filtered
probability space $\filpspace$ satisfying the usual conditions,
with a $p$-dimensional $\mathcal{F}_t$-Brownian motion $W$ and an
adapted $p$-dimensional stochastic process $X$, such that
$X_0=x_0$ a.s.\ and (\ref{eq:SDE}) holds. Let us also recall the
result from Yamada and Watanabe \cite[Theorem~21.14]{Kallenberg}
that (\ref{eq:SDE}) has a \emph{unique strong solution} if and
only if weak existence and pathwise uniqueness holds.

We use the following definitions.
\begin{definition}\label{def:stochinvar}
We call a measurable set $\mathcal{X}\subset\R^p$
\emph{stochastically invariant}, if for all $x_0\in\mathcal{X}$,
there exists a weak solution $(X,W)$ to (\ref{eq:SDE}) with
initial condition $x_0$ such that $X_t\in\mathcal{X}$ almost
surely, for all $t\geq0$.
\end{definition}
\begin{definition}
The SDE~(\ref{eq:SDE}) is called an \emph{affine SDE} with state
space $\mathcal{X}\subset\R^p$ if
\begin{enumerate}
\item it has a unique strong solution;
\item $\mathcal{X}$ is stochastically invariant;
\item the drift $\mu$ and diffusion matrix $\theta=\sigma\sigma^\top$ are affine in $x$ on $\mathcal{X}$, i.e.\
\begin{align}\label{eq:affinemusigma}
\mu(x)=a x +b,\quad\theta(x)= A^0+\sum_{i=1}^p A^i x_i,\mbox{ for
all $x\in\mathcal{X}$},
\end{align}
for some $a\in\R^{p\times p}$, $b\in\R^p$, symmetric
$A^i\in\R^{p\times p}$.
\end{enumerate}
The unique strong solution to an affine SDE is called an affine
diffusion, which corresponds to the definition in
\cite{fm09,part1} in view of \cite[Theorem~2.5]{part1}.
\end{definition}
\medskip

Throughout the next sections we will address the following topics.
Given a state space $\mathcal{X}$, we determine all affine
functions $\mu:\R^p\rightarrow\R^p$ and $\theta:\R^{p\times
p}\rightarrow\R^p$ such that there exists a continuous square root
$\sigma$ of $\theta$ on $\mathcal{X}$ for which (\ref{eq:SDE}) is
an affine SDE. To that end, the following three aspects have to be taken into consideration.
\begin{itemize}
\item
First, it is necessary that $
\mathcal{X}\subset\{\theta\geq0\}$, since
$\theta(x)=\sigma(x)\sigma(x)^\top$ for $x\in\mathcal{X}$.
\item
Second, for stochastic invariance,
one has to impose conditions on $\mu$ and $\theta$ on the boundary
of $\mathcal{X}$, in order to prevent the solution $X$ from
leaving the state space $\mathcal{X}$.
\item
Third, one has to construct a square root $\sigma$ such that
(\ref{eq:SDE}) admits a unique strong solution that stays in
$\mathcal{X}$. Remarkably, for the polyhedral and parabolic state
space we consider, the conditions for stochastic invariance enable
the construction of such a square root $\sigma$, see
Theorems~\ref{th:constructsigma} and~\ref{th:quadr}. For the
conical state space we obtain unique strong solutions by imposing
conditions for invariance of $\{\theta>0\}$, see
Theorem~\ref{th:strongsolconical}.
\end{itemize}

Although obvious, it is worth noting that if $X$ is an affine
diffusion with drift $\mu(x)$, diffusion matrix $\theta(x)$ and
state space $\mathcal{X}$, then $LX+\ell$ is an affine diffusion
with drift $L\mu(L^{-1}(x-\ell))$, diffusion matrix
$L\theta(L^{-1}(x-\ell))L^\top$ and state space
$L\mathcal{X}+\ell$, for non-singular $L\in\R^{p\times p}$,
$\ell\in\R^p$. Therefore, for the tasks as outlined above, it suffices to characterize all affine
diffusions where the state space is in a certain ``canonical''
form (not to be confused with the standard canonical form), thereby obtaining all
remaining diffusions by affine transformations.

\medskip

\begin{remark}\label{rem:generator}
The distribution of $X$ does not change with different choices of
the square root $\sigma$ (as long as $X$ stays in $\mathcal{X}$
for these choices), since it is determined by the generator, which
depends on $\sigma$ only through $\sigma\sigma^\top$. Note however
that if strong existence and uniqueness holds for one particular
choice of $\sigma$, it does not automatically hold for other
choices. For instance take $\theta(x)=1$ in $\R$ and consider the
1-dimensional SDE $\dd X_t = \sigma(X_t)\dd W_t$ with
$\sigma\sigma^\top=\theta$. Existence and uniqueness of a strong
solution holds when we take $\sigma(x)=1$, while one only has a
weak solution for the choice $\sigma(x)=\sgn(x+)$, see
\cite[Example 5.3.5]{Karshr}.
\end{remark}

\subsection*{Matrix notation} The following notation
regarding matrices and vectors is used throughout. Let $p,q\in\N$,
$P=\{1,\ldots,p\}$, $Q=\{1,\ldots,q\}$, $A\in\R^{p\times q}$,
$I\subset P$, $J\subset Q$. Write $I=\{i_1,\ldots,i_{\# I}\}$,
$J=\{j_1,\ldots,j_{\# J}\}$, with $i_1\leq i_2\leq \ldots\leq
i_{\# I}$ and  $j_1\leq j_2\leq\ldots\leq j_{\# J}$. Then $A_{IJ}$
denotes the $(\# I \times \#J)$-matrix with elements
$(A_{IJ})_{kl}=A_{i_k j_l}$.  If $\#I=1$, say $I=\{i\}$, we write
$A_{iJ}$ instead. If $J=Q$ then we write $A_I$ instead. In
particular, $A_i$ denotes the $i$-th row of $A$. The $j$-th column
is denoted by $A^j$ and the transpose of $A$ is denoted by
$A^\top$. The above notation is also used for matrix-valued
functions $\phi$, e.g.\ $\phi^\top(x)$ stands for
$(\phi(x))^\top$.

For $a_1,\ldots,a_p\in\R$ we write $\diag(a_1,\ldots,a_p)$ for the
$p$-dimensional diagonal matrix $D$ with diagonal elements
$D_{ii}=a_i$, $i\in P$. We also write $\diag(a)$ instead, where
$a$ denotes the vector with elements $a_i$, sometimes explicitly
denoted by $a=\textnormal{vec}(a_1,\ldots,a_p)$. For a vector $v$
we write $|v|$ for the vector with elements $|v_i|$ and
analogously $\sqrt{|v|}$ denotes the vector with elements
$\sqrt{|v_i|}$. The identity matrix is written as $\Id$. We write
$[A]$ for the linear span of the row vectors of $A$. If $A$ and
$B$ are two matrices with the same column dimension, then $A\bot
B$ stands for $[A]\bot[B]$, i.e.\ $A_i B_j^\top=0$ for all $i$ and
$j$. The unique positive semi-definite square root of a positive
semi-definite matrix $A$ is denoted by $A^{1/2}$. For a square
matrix $A$ we write $|A|$ for $(AA^\top)^{1/2}$. We will often
make use of the fact that the matrix-valued function $A\mapsto
|A|^{1/2}$ is continuous, which is a consequence of
\cite[Theorem~X.1.1]{bhatia}.

\section{Stochastic invariance}\label{sec:stochinvar}

In this section we obtain necessary and in some cases sufficient
boundary conditions for stochastic invariance, see
Definition~\ref{def:stochinvar}. We first consider stochastic
invariance of a closed convex set $\mathcal{X}\subset\R^p$, for
which we make use of the fact that it can be written as an
intersection of halfspaces, i.e.\
\begin{align}\label{eq:convexX}
\mathcal{X}=\bigcap_{i\in I}\{u_i\geq 0\},
\end{align}
with $I$ some index set and $u_i:\R^p\rightarrow\R:x\mapsto
\gamma_i x +\delta_i$, for some $\gamma_i\in\R^{1\times p}$,
$\delta_i\in\R$. We denote the $i$-th boundary segment
$\mathcal{X}\cap\{u_i=0\}$ of $\partial\mathcal{X}$ with $\XXi$.
The following proposition is partly proved in
\cite[Lemma~B.1]{fm09}. We give a more intuitive proof, involving
an appropriate change of measure.

\begin{prop}\label{prop:necinvar}
Let $\mathcal{X}\subset\R^p$ be a closed convex set given by
(\ref{eq:convexX}) and assume $\mathcal{X}$ is stochastically
invariant. Then necessarily it holds that
\begin{align}
\forall i\in I,\forall x\in\XXi:
\gamma_i\sigma(x)&=0\label{al:theta}\\
\forall i\in I,\forall x\in\XXi:\gamma_i\mu(x)&\geq0\label{al:mu}.
\end{align}
\end{prop}
\begin{proof}
We give a proof by contradiction. Suppose there is an $i\in I$ and
$x_0\in\XXi$ such that (\ref{al:theta}) or (\ref{al:mu}) does not
hold. Then there exists $c\in\R$ such that $\eta$ defined as
$\eta(x)=\gamma_i\mu(x)+c\gamma_i\sigma(x)\sigma(x)^\top\gamma_i^\top$
is negative in $x_0$. Let $(X,W)$ be a weak solution to
(\ref{eq:SDE}) with initial condition $x_0$ on some filtered
probability space $\filpspace$.
There exists a stopping time $\tau_1>0$ such that
$L=\mathcal{E}(c\gamma_i\sigma(X)1_{[0,\tau_1]}\cdot W)$ is a
martingale. Take $T>0$ arbitrarily, then we can change $\PP$ into
an equivalent probability measure $\Q$ on $\mathcal{F}_T$ by $\dd
\Q=L_T\dd \PP$. By Girsanov's Theorem, $W^\Q$ defined by $\dd
W^\Q_t=\dd W_t-c\sigma(X_t)^\top\gamma_i^\top 1_{[0,\tau_1]}(t)\dd
t$, is a Brownian motion under $\Q$ on $[0,T]$. Hence $X$ solves
an SDE under $\Q$ for $t\in[0,T]$, namely
\[
\dd X_t =
(\mu(X_t)+c\sigma(X_t)\sigma(X_t)^\top\gamma_i^\top1_{[0,\tau_1]}(t))\dd
t +\sigma(X_t)\dd W^\Q_t.
\]
Let $\tau_2>0$ be a stopping time such that $(\sigma(X)\cdot
W^\Q)_{t\wedge \tau_2}$ is a $\Q$-martingale and $\eta(X_t)<0$ for
$t\in[0,\tau_2]$, $\Q$-a.s.\ (this is possible since
$\eta(X_0)=\eta(x_0)<0$ and $t\mapsto X_t$ and $\eta$ are
continuous, $\PP$-a.s., whence $\Q$-a.s.). Write
$\tau=\tau_1\wedge\tau_2$. Then it holds that $\tau>0$, $\Q$-a.s.,
and therefore
\begin{align*}
\E_\Q u_i(X_{T\wedge\tau}) &= \delta_i+\gamma_i\E_\Q
X_{T\wedge\tau} =u_i(x_0)+\E_\Q \int_0^{T\wedge\tau}\eta(X_t)\dd
t<0.
\end{align*}
This implies $\Q(\forall t\in[0,T]:u_i(X_t)\geq0)<1$ and by
equivalence of $\PP$ and $\Q$ also $\PP(\forall
t\in[0,T]:u_i(X_t)\geq0)<1$, which contradicts the stochastic
invariance of $\mathcal{X}$.
\end{proof}
\medskip

If $\mathcal{X}\subset\R^p$ given by (\ref{eq:convexX}) is a
convex polyhedron, then $I$ is finite, say $I=\{1,\ldots,q\}$, for
some $q\in\N$. We write $\gamma\in\R^{q\times p}$ for the matrix
with row vectors $\gamma_i$,  $\delta\in\R^q$ for the vector with
elements $\delta_i$ and $u:\R^p\rightarrow\R^q:x\mapsto \gamma
x+\delta$. The necessary conditions obtained in
Proposition~\ref{prop:necinvar} are sufficient when $X$ is a
convex polyhedron congruent to the canonical space
$\R^m_{\geq0}\times\R^{p-m}$.
\begin{prop}\label{prop:sufinvar}
Let $\mathcal{X}\subset\R^p$ be a convex polyhedron given by
(\ref{eq:convexX}) with $I=\{1,\ldots,q\}$ for some $q\in\N$ and
suppose $\gamma$ has full row-rank. Then $\mathcal{X}$ is
stochastically invariant if and only if (\ref{al:theta}) and
(\ref{al:mu}) hold.
\end{prop}
\begin{proof}
This follows by \cite[Remark 3.6]{tappe09}, since $\mathcal{X}$ is
congruent to the standard canonical state space, as $\gamma$ has
full row-rank.
\end{proof}
\medskip

Proposition~\ref{prop:sufinvar} has the following corollary, which
will turn out to be the building block for proving stochastic
invariance for affine diffusions with general polyhedral state
space, as considered in Section~\ref{sec:polyhedral}.
\begin{cor}\label{cor:stochinvarcan}
Let $a\in\R^{p\times p}$, $b\in\R^p$, $c\in\R^p_{\geq0}$ and
suppose the drift of (\ref{eq:SDE}) is given by
\begin{align}\label{eq:affinemu}
\mu(x)= a x +b,
\end{align}
and the diffusion coefficient by
$\sigma(x)=\diag(\sqrt{c_1|x_1|},\ldots,\sqrt{c_p|x_p|})$, for all
$x\in\R^p$. Then $\R^p_{\geq0}$ is stochastically invariant if and
only if
\begin{align}\label{eq:aijgeq0}
\forall i\in P,\forall j\in P\backslash\{i\}: a_{ij}\geq0\mbox{
and } b_i\geq0.
\end{align}
\end{cor}

To obtain conditions for stochastic invariance of a quadratic
state space, we note that if $\mathcal{X}\subset\R^p$ is a closed
convex set with $\partial\mathcal{X}\subset\{\Phi=0\}$ for some
$C^1$-function $\Phi:\R^p\rightarrow\R$, then we can take
(\ref{eq:convexX}) with $I=\partial\mathcal{X}$,
$\gamma_i=\nabla\Phi(i)$ and $\delta_i=-\gamma_i i$, for $i\in I$,
where we write $\nabla\Phi$ for the gradient of $\Phi$ (written as
a row vector). The necessary conditions for invariance
(\ref{al:theta}) and (\ref{al:mu}) in this case read
\begin{align*}
\forall x\in\partial\mathcal{X}: \nabla\Phi(x) \sigma(x)&=0\\
\forall x\in\partial\mathcal{X}:\nabla\Phi(x) \mu(x)&\geq 0.
\end{align*}
The next proposition gives necessary and sufficient conditions for
a particular case. We write $\nabla^2\Phi=\nabla^\top\nabla\Phi$
for the Hessian of a $C^2$-function $\Phi$.
\begin{prop}\label{prop:invargeneral}
Let $\Phi:\R^p\rightarrow\R:x\mapsto x_1-f(x_{P\backslash\{1\}})$
for some $C^2$-function $f:\R^{p-1}\rightarrow \R$. Then
$\mathcal{X}=\{\Phi\geq0\}$ is invariant for (\ref{eq:SDE}) if and
only if
\begin{align}
  \nabla\Phi(x) \sigma(x)&=0\label{con:invarPsi2}\\
\nabla\Phi(x) \mu(x)&\geq -\half\tr(
\nabla^2\Phi(x)\sigma(x)\sigma(x)^\top)\label{con:invarPsi1},
\end{align}
for all $x\in\partial\mathcal{X}=\{\Phi=0\}$.
\end{prop}
\begin{proof}
Let $(X,W)$ be a weak solution to (\ref{eq:SDE}) on some filtered
probability space $\filpspace$ with initial condition
$x_0\in\mathcal{X}$. It holds that $X_t\in\mathcal{X}$ if and only
if
$(\Phi(X_t),X_{P\backslash\{1\},t})\in\R_{\geq0}\times\R^{p-1}$.
It\^{o}'s formula gives
\begin{align*}
\dd \Phi(X_t) &= \nabla\Phi(X_t) \dd X_t +
\half\tr(\nabla^2\Phi(X_t)\dd \langle
X_t\rangle)\\
&= (\nabla\Phi(X_t) \mu(X_t)+
\half\tr(\nabla^2\Phi(X_t)\sigma(X_t)\sigma(X_t)^\top))\dd t \\&+
\nabla\Phi(x) \sigma(X_t)\dd W_t.
\end{align*}
Applying Proposition~\ref{prop:sufinvar} gives the result.
\end{proof}
\medskip

Note that if $\mathcal{X}$ is convex in
Proposition~\ref{prop:invargeneral}, equivalently $f$ is a convex
function, then $\nabla^2 f$ is positive semi-definite on its
domain, whence $\nabla^2 \Phi$ is negative semi-definite and it
follows that
\[
\tr(\nabla^2\Phi(X_t)\sigma(x)\sigma(x)^\top)=\tr(\sigma(x)^\top
\nabla^2\Phi(x)\sigma(x))\leq 0,\mbox{ for all }x\in\R^p.
\]
Thus condition~(\ref{con:invarPsi1}) is stronger than
(\ref{al:theta}), whence (\ref{al:theta}) and (\ref{al:mu}) are in
general not sufficient for stochastic invariance. For more results
we refer to \cite{Daprato07} and \cite{Milian95}.

We conclude this section by giving sufficient conditions for
stochastic invariance of an open set. The idea behind the proof of
the following result is taken from \cite{pfaffel09}.
\begin{prop}\label{prop:nablaphi}
Let $\Phi:\R^p\rightarrow\R$ be a $C^2$-function and suppose
$\mathcal{X}^\circ$ is a connected component (i.e.\ maximal
connected subset) of $\{\Phi>0\}$. Then $\mathcal{X}^\circ$ is
invariant for (\ref{eq:SDE}) if there exists an open neighborhood
$O$ of $\partial\mathcal{X}$ such that
\begin{align}\label{eq:nablaphi}
\nabla\Phi(x)\mu(x)\geq-\half\tr(\nabla^2\Phi(x)\theta(x))+\half\Phi(x)^{-1}\nabla\Phi(x)\theta(x)(\nabla\Phi(x))^\top,\end{align}
for all $x$ in $O\cap\mathcal{X}^\circ$, where we write
$\theta=\sigma\sigma^\top$.
\end{prop}
\begin{proof}
Let $(X,W)$ be a weak solution to (\ref{eq:SDE}) on some filtered
probability space $\filpspace$ with initial condition
$x_0\in\mathcal{X}^\circ$. By It\^{o}'s formula it holds for
$t<\tau_0:=\inf\{s\geq0:\Phi(X_s)=0\}$ that
\begin{align*}
\dd\log\Phi(X_t)&=\Phi(X_t)^{-1}\dd\Phi(X_t)-\half\Phi(X_t)^{-2}\dd\langle\Phi(X)\rangle_t\\
&=\Phi(X_t)^{-1} (f(X_t)\dd t+\sigma(X_t)\dd W_t),
\end{align*}
where
\[
f(x)=\nabla\Phi(x)\mu(x)+\half(\tr(\nabla^2\Phi(x)\theta(x))-\Phi(x)^{-1}\nabla\Phi(x)\theta(x)(\nabla\Phi(x))^\top).
\]
Suppose $A:=\{\tau_0<\infty\}$ has positive probability. Since $X$
does not explode, it holds for each $\omega\in A$ that there
exists $\varepsilon>0$ such that $X_t(\omega)\in
O\cap\mathcal{X}^\circ$ for
$\tau_0(\omega)-\varepsilon<t<\tau_0(\omega)$, whence
$f(X_t(\omega))\geq 0$. Therefore, $\int_0^t \Phi(X_s)^{-1}
f(X_s)\dd s$ does not tend to $-\infty$ on $A$ for
$t\uparrow\tau_0$. This yields
\[
\lim_{t\uparrow\tau_0}\int_0^t \Phi(X_s)^{-1}\sigma(X_s)\dd
W_s=-\infty,
\]
which is impossible, by the arguments of
\cite[Proposition~3.4]{pfaffel09}.
\end{proof}
\begin{remark}
Note that (\ref{eq:nablaphi}) implies (\ref{con:invarPsi2}) and is
stronger than (\ref{con:invarPsi1}), as $\theta=\sigma\sigma^\top$
is positive semi-definite.
\end{remark}
\medskip

Proposition~\ref{prop:nablaphi} yields tractable conditions for
stochastic invariance of an open set in case of an affine SDE.
\begin{prop}\label{prop:nablaphi2}
Consider the situation of Proposition~\ref{prop:nablaphi} and
suppose $\mu$ and $\theta:=\sigma\sigma^\top$ are affine functions
given by (\ref{eq:affinemusigma}). Then $\mathcal{X}^\circ$ is
invariant for (\ref{eq:SDE}) if
\begin{align}
\nabla\Phi(x) \theta(x)&=\Phi(x) v^\top,\,\mbox{ for some constant $v\in\R^p$},\label{al:phiv}\\
\nabla\Phi(x)(\mu(x)-\half\sum_{i=1}^p (A^i)^i)&\geq 0,\,\mbox{
for all }x\in \mathcal{X}\label{al:phiv2},
\end{align}
where $(A^i)^i$ denotes the $i$-th column of the matrix $A^i$.
\end{prop}
\begin{proof}
In view of Proposition~\ref{prop:nablaphi} it suffices to show
that
\[
\tr(\nabla^2\Phi(x)\theta(x))-\Phi(x)^{-1}\nabla\Phi(x)\theta(x)(\nabla\Phi(x))^\top=-\nabla\Phi(x)\sum_{i=1}^p
(A^i)^i,
\]
under the assumption that $\nabla\Phi(x) \theta(x)=\Phi(x)v^\top$,
for some constant vector $v\in\R^p$. Differentiating the right-
respectively the left-hand side of the latter and applying the
product rule yields
\[
\nabla^\top\Phi(x)v^\top=\nabla^\top (\nabla\Phi(x) \theta(x))=
\nabla^2\Phi(x)\theta(x)+\nabla\Phi(x)\nabla^\top\theta(x),
\]
where $\nabla\Phi(x)\nabla^\top\theta(x)$ is short-hand notation
for the matrix with row vectors
$\nabla\Phi(x)\displaystyle\frac{\partial}{\partial
x_i}\theta(x)$. Hence
\[
\tr(\nabla^2\Phi(x)\theta(x))=\nabla\Phi(x)v-\tr(\nabla\Phi(x)
\nabla^\top\theta(x)).
\]
The result follows, since $\nabla\Phi(x)
v=\Phi(x)^{-1}\nabla\Phi(x)\theta(x)(\nabla\Phi(x))^\top$ and
\[
\tr(\nabla\Phi(x)
\nabla^\top\theta(x))=\sum_{i=1}^p\nabla\Phi(x)\frac{\partial\theta^i(x)}{\partial
x_i}=\nabla\Phi(x)\sum_{i=1}^p (A^i)^i.
\]
\end{proof}
\section{Polyhedral state space}\label{sec:polyhedral}
\subsection{General diffusion matrix} Throughout this section, we
assume $\mathcal{X}$ is a polyhedron given by
\begin{align}\label{eq:convexpolyX}
\mathcal{X}=\bigcap_{i=1}^q\{u_i\geq0\},
\end{align}
for some affine function $u:\R^q\rightarrow\R^p:x\mapsto \gamma x
+\delta$, $\gamma\in\R^{q\times p}$, $\delta\in\R^q$, some
$q\in\N$. We write $Q=\{1,\ldots,q\}$ and assume $Q$ is
\emph{minimal} in the sense that $\bigcap_{i\in
Q'}\{u_i\geq0\}\neq\mathcal{X}$ for all $Q'\subset Q$ with $Q'\neq
Q$. In addition we assume $\mu$ is of the form (\ref{eq:affinemu})
and we are given an affine function $\theta$ by
\begin{align}\label{eq:affinetheta}
\theta:\R^p\rightarrow\R^{p\times p}:x\mapsto A^0+\sum_{i=1}^p A^i
x_i,
\end{align}
for some symmetric $A^i\in\R^{p\times p}$ and we assume
\begin{align}\label{eq:XinD}
\emptyset\neq\mathcal{X}^\circ\subset\mathcal{X}\subset\{x\in\R^p:\theta(x)\geq0\}=:\mathcal{D}.
\end{align}
We will often make use of the fact that for a symmetric matrix $S$
and vector $v$ it holds that $Sv=0$ is equivalent with $v^\top Sv
=0$. In particular if $\theta=\sigma\sigma^\top$ on $\mathcal{X}$,
then (\ref{al:theta}) is equivalent to
\[
\forall i\in I,\forall x\in\XXi: \gamma_i\theta(x)=0.
\]

In Section~\ref{sec:stochinvar} we have derived necessary boundary
conditions on $\mu$ and $\theta$ to have $\mathcal{X}$
stochastically invariant. We show in this subsection that for a
polyhedral state space, these conditions are also sufficient for
the existence of a square root $\sigma$ of $\theta$ on
$\mathcal{X}$ such that (\ref{eq:SDE}) is an affine SDE. In the
next proposition we prove that $\theta$ can be transformed in a
block-diagonal form, which we use in the proof of
Theorem~\ref{th:constructsigma} to construct the square root
$\sigma$. The proposition improves upon the results given in the
appendix of \cite{dk96} and generalizes \cite[Lemma~7.1]{fm09}.
\begin{prop}\label{prop:LthetaLT}
Let $\mathcal{X}\subset\R^p$ be a convex polyhedron given by
(\ref{eq:convexpolyX}) and satisfying (\ref{eq:XinD}) and let
$\theta$ be given by (\ref{eq:affinetheta}). Assume
\begin{align}\label{eq:cond}
\forall i\in Q,\forall x\in\XXi: \gamma_i\theta(x)=0.
\end{align}
Then there exists a non-singular $L\in\R^{p\times p}$ and a vector
$\ell\in\R^p$ such that
\begin{align}\label{eq:LthetaLT}
L\theta(L^{-1}(x-\ell))L^\top=\begin{pmatrix}
    \diag(x_M,0_N) & 0 \\
    0 & \Psi(x_{M\cup N})
  \end{pmatrix},
\end{align}
for some index sets $M=\{1,\ldots,m\}$, $N=\{m+1,\ldots,m+n\}$ and
affine function $\Psi$. In addition, we have
\begin{align}\label{eq:statespace}
L\mathcal{X}+\ell=\R^m_{\geq 0} \times \mathcal{C} \times
\R^{p-m-n},
\end{align}
for some convex polyhedron
$\mathcal{C}=\bigcap_{i=1}^{q-m}\{\wu_i\geq0\}\subset\R^n_{\geq0}$
with $\wu_i(x)=x_i$ for $i\leq n$.
\end{prop}
\begin{proof} We divide the proof into a couple of steps.

\emph{Step 1.} There exists $B\in\R^{q\times p}$ such that for
$i\in Q$ it holds that
\begin{align}
\gamma_i\theta(x)&=B_i u_i(x),\label{con:B1}\\
B_i\gamma_i^\top&>0, \mbox{ if }B_i\not=0\label{con:B2}\\
B_i\gamma_j^\top&=0,\mbox{ for $j\in
Q\backslash\{i\}$}.\label{con:B3}
\end{align}
This is shown as follows. Fix $i\in Q$. By (\ref{eq:cond}) and
Proposition~\ref{prop:uiscv} there exists $B_i\in\R^{1\times p}$
such that $\gamma_i\theta^j(x)=B_{ij}u_i(x)$ for $j\in P$. By
assumption there exists $x_0\in\mathcal{X}^\circ$. It holds that
$u_i(x_0)>0$, so we can write
\begin{align}\label{eq:Bi}
B_i = u_i(x_0)^{-1}\gamma_i\theta(x_0).
\end{align}
By positive semi-definiteness of $\theta(x_0)$, it holds that
$B_i\gamma_i^\top\geq0$. We have $B_i\gamma_i^\top=0$ if and only
if $\gamma_i\theta(x_0)=0$, i.e.\ $B_i=0$. This yields
(\ref{con:B2}). Moreover, if $j\in Q\backslash\{i\}$, then by
symmetry of $\theta$ it holds that
\[
B_i\gamma_j^\top
u_i(x)=\gamma_i\theta(x)\gamma_j^\top=\gamma_j\theta(x)\gamma_i^\top=B_j\gamma_i^\top
u_j(x).
\]
This implies $B_i\gamma_j^\top=0$, since $Q$ is minimal.

\emph{Step 2.} It holds that
\begin{align}
\rank\gamma_M&=\rank B_M = m\label{con:M1}\\
\R^{1\times p}&=[\gamma_M]\oplus[B_M]^\bot\label{con:M2},
\end{align}
where
\begin{align}
M=\{i\in Q:B_i\not = 0\}\label{eq:defM}.
\end{align}
Indeed, by (\ref{con:B3}) we have
\begin{align}\label{eq:bmgm}
B_M\gamma_M^\top=\diag(B_1\gamma_1^\top,\ldots,B_m\gamma_m^\top),
\end{align}
which has full rank $m$ by (\ref{con:B2}). This implies
(\ref{con:M1}) as well as $[\gamma_M]\cap[B_M]^\bot=\{0\}$, which
yields (\ref{con:M2}).

\emph{Step 3.} It holds that
\begin{align}
[B_M]^\bot&=[\gamma_N]\oplus[\eta],\label{con:step3}\\
[\gamma_{Q\backslash M}]&=[\gamma_N],
\end{align}
where $N\subset Q\backslash M$ is such that $\gamma_{M\cup N}$ has
full rank, $\rank\gamma_{M\cup N}=\rank\gamma$, for some
$\eta\in\R^{(p-m-n)\times p}$, with $n=\# N$. The first equality
follows immediately from (\ref{con:M1}) and (\ref{con:M2}), while
the second holds since $[\gamma_{Q\backslash M}]\subset[B_M]^\bot$
by (\ref{con:B3}) and $[\gamma_{M\cup N} ]=[\gamma]$.

\emph{Step 4.} Permute indices such that $M=\{1,\ldots,m\}$,
$N=\{m+1,\ldots,m+n\}$. We define $L\in\R^{p\times p}$ by
\begin{align}
L_{M\cup N}&=\gamma_{M\cup N}\label{eq:propertyL}\\
L_{P\backslash (M\cup N)}&= \eta\label{eq:propertyL2}.
\end{align}
Then $L$ is non-singular, by (\ref{con:M2}) and (\ref{con:step3}).
Moreover, the previous steps yield
\[
 L\theta(x)L^\top =
  \begin{pmatrix}
    \diag(c_1u_1(x),\ldots,c_m u_m(x),0_N) & 0 \\
    0 & \Phi(x)
  \end{pmatrix},
\]
with $c_i=B_i\gamma_i^\top$ and $\Phi$ defined by $\Phi(x)=
\eta\theta(x)\eta^\top$.
It holds that $c_i>0$ by (\ref{con:B2}). By rescaling $u_i$ we may
assume $c_i=1$. Then we take $\ell\in\R^p$ such that $\ell_{M\cup
N }=\delta_{M\cup N}$.

\emph{Step 5.} It remains to show that we can write $\Phi(x)$ as
an affine function of $u_{M\cup N}(x)$. This is an immediate
consequence of the assumption that
$\mathcal{X}\subset\{\theta\geq0\}$, as this yields that
$\theta(x)$ only depends on $u(x)$, which is a function of
$u_{M\cup N}(x)$.
%
\end{proof}
\begin{remark}
For the canonical state space as treated in \cite{dfs03} and
\cite{fm09} one has (\ref{eq:statespace}) with
$\mathcal{C}=\R^n_{\geq0}$ and
\[
\Psi(x_{M\cup N})=B^0+\sum_{i=1}^{m+n}B^i x_i,
\]
with $B^i$ positive semi-definite.
\end{remark}
\begin{theorem}\label{th:constructsigma}
Let
 $\mathcal{X}\subset\R^p$ be given by
(\ref{eq:convexpolyX}) and satisfying (\ref{eq:XinD}). There
exists an affine SDE with drift $\mu$, diffusion matrix $\theta$
and polyhedral state space $\mathcal{X}$ if and only if
\begin{align}
\forall i\in Q,\forall x\in\XXi: \gamma_i\theta(x)&=0,\label{al:cond1}\\
\forall i\in Q,\forall
x\in\XXi:\gamma_i\mu(x)&\geq0\label{al:cond2}.
\end{align}
\end{theorem}
\begin{proof}
The ``only if''-part is Proposition~\ref{prop:necinvar}. We prove
the ``if''-part.

Assume (\ref{al:cond1}) and (\ref{al:cond2}). By
Proposition~\ref{prop:LthetaLT} we may assume
\begin{align}\label{eq:canonicaltheta}
\theta(x)=\begin{pmatrix}
    \diag(x_M,0_N) & 0 \\
    0 & \Psi(x_{M\cup N})
  \end{pmatrix},\mbox{ for $x\in \mathcal{X}$}
\end{align}
and
\begin{align}\label{eq:canonicalX}
\mathcal{X}=\R^m_{\geq 0} \times \mathcal{C} \times \R^{p-m-n},
\end{align}
for some index sets $M=\{1,\ldots,m\}$, $N=\{m+1,\ldots,m+n\}$, an
affine function $\Psi$ and some convex polyhedron
$\mathcal{C}\subset\R^n_{\geq0}$, such that $u_{M\cup
N}(x)=x_{M\cup N}$. As a square root of $\theta$ on $\mathcal{X}$
we take
\begin{align}\label{eq:constrsigma}
\sigma(x)= |\theta(x)|^{1/2}=\begin{pmatrix}
    \diag(\sqrt{|x_M|},0_N) & 0 \\
    0 & |\Psi(x_{M\cup N})|^{1/2}
  \end{pmatrix}.
\end{align}
It remains to show that there exists a unique strong solution $X$
to (\ref{eq:SDE}) which stays in $\mathcal{X}$ for all
$x_0\in\mathcal{X}$.

Let $x_0\in\R^p$ be arbitrary. By continuity of the coefficients
$\mu$ and $\sigma$ and satisfaction of the linear growth condition
(\ref{eq:lingrow}), there exists a weak solution $(X,W)$ to
(\ref{eq:SDE}) on some filtered probability space $\filpspace$. To
show stochastic invariance of $\mathcal{X}$, note that by
condition~(\ref{al:cond2}) and Proposition~\ref{prop:driftpos}
there exist $\wa\in\R^{q\times q}$ with $\wa_{ij}\geq0$ for
$i,j\in Q$ with $i\not=j$, and $\wb\in\R^q_{\geq0}$ such that
$\gamma\mu(x)=\wa u(x) +\wb$. Hence $u(X_t)$ satisfies the
$q$-dimensional SDE
\begin{equation}\label{eq:SDEu}
\begin{split}
\dd u(X_t)
    &=(\wa u(X_t)+\wb)\dd t +
    \diag(\sqrt{|X_{M,t}|},0_{Q\backslash
    M})\dd \wW_t,
\end{split}
\end{equation}
with initial condition $u(x_0)$, where $\wW$ is a $q$-dimensional
Brownian motion with $\wW_M=W_M$ (possibly after extending
$\Omega$). For this SDE strong existence and uniqueness holds by
\cite[Theorem~1]{YamadaI}. Together with
Corollary~\ref{cor:stochinvarcan} this yields $u(X_t)\in
\R^q_{\geq 0}$, $\PP$-a.s., for all $t\geq 0$ if
$u(x_0)\in\R^q_{\geq0}$. In other words, $X_t\in\mathcal{X}$,
$\PP$-a.s., for all $t\geq 0$ if $x_0\in\mathcal{X}$. Thus
$\mathcal{X}$ is stochastically invariant.

We now show pathwise uniqueness for (\ref{eq:SDE}). Therefore, let
$(\wX,W)$ be another solution on the same filtered probability
space $\filpspace$ with initial condition $x_0$. Pathwise
uniqueness for (\ref{eq:SDEu}) implies that $X_{M\cup N}=u_{M\cup
N}(X)=u_{M\cup N}(\wX)=\wX_{M\cup N }$ a.s. Write
$R={P\backslash(M\cup N)}$. Since $X_{R}$ does not appear in the
diffusion part of the SDE, we have
\[
\dd (X_{R,t}-\wX_{R,t})=a_{RR}(X_{R,t}-\wX_{R,t})\dd t,\quad
X_{R,0}-\wX_{R,0}=0.
\]
So $X_{R}-\wX_{R}$ solves a linear ODE which has 0 as its unique
solution, whence $X_{R}=\wX_{R}$ a.s.\ and the result follows.
\end{proof}
\medskip

By an inspection of the proof of Theorem~\ref{th:constructsigma}
we see that if $X$ solves an affine SDE with polyhedral state
space $\mathcal{X}$ given by (\ref{eq:statespace}), then $u(X)$
solves the affine SDE~(\ref{eq:SDEu}) with \emph{admissible}
parameters in the sense of \cite{ds00} (that is, $\wa$ has
non-negative off-diagonal elements and $\wb\in\R^q_{\geq0}$).
Extending the dimension by considering $(u(X),X_R)$ instead of
$X$, we get another affine SDE with state space
$u(\mathcal{X})\times\R^r$ (with $r=\# R$) and diffusion matrix
$\wtheta$ (say). Now if $\theta$ can be written as an affine
transformation of $u$ with positive semi-definite matrices, then
$\wtheta$ is positive semi-definite on the whole of
$\R^q_{\geq0}\times\R^{r}$. In that case,
$u(\mathcal{X})\times\R^r$ can be enlarged to the canonical state
space $\R^q_{\geq0}\times\R^{r}$ and the resulting SDE is of the
canonical form as in \cite{fm09}. We elaborate on this in the next
subsection.
\subsection{Diagonalizable diffusion matrix}
In \cite{cfk08} it is shown that the diffusion matrix $\theta(x)$
of an affine SDE with a standard canonical state space cannot be
diagonalized in general. That is, there does not exist a
non-singular matrix $L$ such that $L\theta(L^{-1} x)L^\top$ is
diagonal. In this subsection we show that for a large class of
polyhedral state spaces, including the standard canonical state
space, diagonalization of the diffusion matrix is still possible
in a different way, by extending the dimension. We also provide
sufficient conditions for this as well as we give an example of an
affine diffusion whose diffusion matrix is not diagonalizable.
\begin{prop}\label{prop:phipos}
Let $X$ be an affine diffusion with drift $\mu$, diffusion matrix
$\theta$ and polyhedral state space $\mathcal{X}$. Then $X$ is in
distribution equal to an affine transformation of an affine
diffusion with \emph{diagonal} diffusion matrix and canonical
state space of the form $\R^m_{\geq0}\times\R^{p-m}$ if and only
if
\begin{align}\label{eq:thetapossum}
\theta(x)=B^0+\sum_{i=1}^q B^i u_i(x),\,\mbox{ for some positive
semi-definite $B^i$}.
\end{align}
\end{prop}
\begin{proof}
We first prove the ``only if''-part. Suppose $X=L\wX+\ell$ for
some matrix $L$ and vector $\ell$, where $\wX$ is an affine
diffusion with diagonal diffusion matrix $\wtheta$ and polyhedral
state space $\widetilde{\mathcal{X}}$. Then it holds that
$L\wtheta(x)L^\top=\theta(Lx+\ell)$ and
$\mathcal{X}=L\widetilde{\mathcal{X}}+\ell$. Since $\wtheta(x)$ is
diagonal, we have
\begin{align}\label{eq:thetaLxplusell}
\theta(Lx+\ell)= L\wtheta(x)L^\top=\sum_{i} d_i(x) L^i (L^i)^\top,
\end{align}
for some affine real-valued functions $d_i$. Note that $L^i
(L^i)^\top\geq0$ and $d_i(x)\geq0$ for
$x\in\widetilde{\mathcal{X}}$. To show that $\theta$ is of the
form (\ref{eq:thetapossum}), it suffices to write $d_i$ as
$d_i(x)=\sum_j \lambda_{ij} u_i(Lx+\ell) + c_i$ for non-negative
$\lambda_{ij}$ and $c_i$. Note that (\ref{eq:thetaLxplusell})
yields $d_i(x+y)=d_i(x)$ for $y\in\ker L$ (otherwise, replace
$d_i(x)$ by $d_i(\Pi (x))$, where $\Pi$ denotes the projection
onto $(\ker L)^\top$). Hence the affine map
\[
\phi_i: Lx+\ell\mapsto d_i(x)
\]
is well-defined and since $\phi_i(y)\geq0$ for
$y\in\mathcal{X}=L\widetilde{\mathcal{X}}+\ell$,
Proposition~\ref{prop:sht} yields the result.

Next we prove the ``if''-part. Suppose (\ref{eq:thetapossum})
holds. By Proposition~\ref{prop:LthetaLT} we may assume $\theta$
is of the form (\ref{eq:canonicaltheta}) with
\[
\Psi(x_{M\cup N})= \Lambda^0+\sum_{i=1}^q \Lambda^i u_i(x),
\]
for some positive semi-definite $\Lambda^i$. In this case
$(u(X),X_{P\backslash (M\cup N)})$ is an affine diffusion where
$u(X)$ satisfies (\ref{eq:SDEu}). Therefore, we assume without
loss of generality that $u(x)=x_{M\cup N}$, so that $Q=M\cup N$,
$q=m+n$ and $\mathcal{X}=\R^{q}_{\geq0}\times\R^{p-q}$. Since
$\Lambda^i\geq0$, its unique positive semi-definite square root
$(\Lambda^i)^{1/2}$ exists. We write
\begin{align*}
\Lambda^{1/2}&=
  \begin{pmatrix}
    (\Lambda^0)^{1/2} & (\Lambda^1)^{1/2} & \ldots & (\Lambda^q)^{1/2}
  \end{pmatrix}\\
  w(x_{Q})&=\vect(1_Q,x_1 1_Q,\ldots,x_q 1_Q).
\end{align*}
Now note that
\[
\Psi(x_{M\cup N})=\Lambda^{1/2}
\diag(w(x_{Q}))(\Lambda^{1/2})^\top,
\]
so $X$ is an affine diffusion with drift $\mu$ and diffusion
matrix
\[
  \theta(x)=\begin{pmatrix}
    \Id & 0 \\
    0 & \Lambda^{1/2}
  \end{pmatrix}
  \begin{pmatrix}
    \diag(x_M,0_{N}) & 0 \\
    0 & \diag(w(x_Q))
  \end{pmatrix}  \begin{pmatrix}
    \Id & 0 \\
    0 & \Lambda^{1/2}
  \end{pmatrix}^\top.
\]
We can diagonalize this by extending the dimension. Define a
non-singular square matrix $T$ by
\[
T=  \begin{pmatrix}
    \Lambda^{1/2} & \Id \\
    \Id & 0
  \end{pmatrix},
\]
and let $\wX$ be an affine diffusion with drift $(\mu(x),0)$ and
diffusion matrix
\[
  \begin{pmatrix}
    \Id & 0 \\
    0 & T
  \end{pmatrix}
  \begin{pmatrix}
    \diag(x_M,0_{N}) & 0 \\
    0 & \diag(w(x_{Q}),0_Q)
  \end{pmatrix}  \begin{pmatrix}
    \Id & 0 \\
    0 & T
  \end{pmatrix}^\top,
\]
and with the law of $\wX_P$ equal to the law of $X_0$. Then
\[
  \begin{pmatrix}
    \Id & 0 \\
    0 & T^{-1}
  \end{pmatrix}\wX
\]
solves an affine SDE with diagonal diffusion matrix $
\diag(x_M,0_N,w(x_Q),0_Q) $ and it is easy to check that $\wX_P$
satisfies an affine SDE with drift $\mu$, diffusion matrix
$\theta$ and initial condition the law of $X_0$. Hence $\wX_P$ is
in distribution equal to $X$ in view of
Remark~(\ref{rem:generator}), which yields the result.
\end{proof}
\medskip

The relevance of diagonalizable diffusion matrices $\theta$ is
elucidated in affine term structure models. In such models the
short rate is an affine transformation of an affine diffusion $X$,
the state factor. In view of Proposition~\ref{prop:phipos}, for an
unobservable state factor $X$, we may assume without loss of
generality that $\theta$ is diagonal when it is of the form
(\ref{eq:thetapossum}). In particular this applies to affine
diffusions in canonical form. This observation reveals that affine
diffusions with canonical state space and a non-diagonal diffusion
matrix have the same potential as those with a diagonal diffusion
matrix, which answers the implicit question in the concluding
section of \cite{cfk08}.

The following example shows that there exist affine diffusions
with non-canonical polyhedral state space that do not meet the
criteria of Proposition~\ref{prop:phipos}.
\begin{example}
Consider the polyhedron
$\mathcal{X}=\mathcal{C}\times\R^{2}\subset\R^4$, where we take
$\mathcal{C}=\bigcap_{i=1}^3\{u_i\geq0\}\subset\R^2_{\geq0}$, with
$u_1(x)=x_1$, $u_2(x)=x_2$, $u_3(x)=x_1+x_2-\frac{3}{2}$. Let
$\mu$ and $\theta$ be given as
\[
\mu_N(x)=
  \begin{pmatrix}
    -x_1+1 \\
    -x_2+1
  \end{pmatrix},\quad \mu_{P\backslash N}\mbox{ arbitrary },\quad
\theta(x)=
  \begin{pmatrix}
    \diag(0_N) & 0  \\
    0 & \Phi(x)
  \end{pmatrix},
\]
with $N=\{1,2\}$ and
\[
\Phi(x)=
  \begin{pmatrix}
    x_1+\half & 1 \\
    1 & x_2+\half
  \end{pmatrix}.
\]
Then $\mu$ and $\theta$ fulfil (\ref{al:cond1}) and
(\ref{al:cond2}), so by Theorem~\ref{th:constructsigma} there
exists an affine diffusion with state space $\mathcal{X}$, drift
$\mu$ and diffusion matrix $\theta$. However, one can show that
$\theta$ is not of the form (\ref{eq:thetapossum}).
\end{example}
%
%
\medskip

We now give sufficient conditions for (\ref{eq:thetapossum}). In
Proposition~\ref{prop:gammafullrank} below we prove that
(\ref{eq:thetapossum}) does not only hold under full-row rankness
of $\gamma$ (equivalent to the canonical state space) but also
under the weaker condition of full row-rankness of
$\begin{pmatrix}\delta &\gamma\end{pmatrix}$ and an additional
assumption.
\begin{prop}\label{prop:gammafullrank}  Let $\mathcal{X}\subset\R^p$ be given by
(\ref{eq:convexpolyX}) and satisfying (\ref{eq:XinD}). Suppose
that either
\begin{enumerate}[(i)]
\item $\gamma$ has full row-rank or that
\item\label{item:triangle} $\begin{pmatrix}
    \delta &\gamma
  \end{pmatrix}$ has full row-rank and for all $i\in Q$,
$x\in\R^p$ it holds that $u_j(x)=0$ for all $j\in
Q\backslash\{i\}$ implies $u_i(x)\geq 0$.
\end{enumerate}
 Then
(\ref{eq:thetapossum}) holds.
\end{prop}
\begin{proof} The case when $\gamma$ has full row-rank is easy, since the state space can be transformed in the canonical form
$\R^m_{\geq0}\times\R^{p-m}$. We consider the case when $\gamma$
has not full row-rank. First we assume $\delta_i\geq0$ for all
$i\in Q$.

Let $\Gamma\in\R^{(p+1)\times(p+1)}$ be a non-singular matrix such
that $\Gamma_{Q}=\begin{pmatrix}
    \delta &\gamma
  \end{pmatrix}$.
  We extend the
  dimension of $\delta$ and $\gamma$ by writing
$ \Gamma=\begin{pmatrix}
    \delta &\gamma
  \end{pmatrix}$.
We also write $u_i(x)=\gamma_i x +\delta_i$, $x\in\R^p$, for all
$i\leq p+1$, in other words
\[
u_i(x)=\Gamma_i \begin{pmatrix}
    1 \\
    x
  \end{pmatrix}.
\]
Let $N=\Gamma^{-1}$. Then we have
\[
  \begin{pmatrix}
    1 \\
    x
  \end{pmatrix}=\Gamma^{-1}\Gamma
  \begin{pmatrix}
    1 \\
    x
  \end{pmatrix}=\sum_{i=1}^{p+1} N^i u_i(x),\mbox{ for $x\in\R^p$}.
\]
Let $\sA=
  \begin{pmatrix}
    A^0 & A^1 &\ldots & A^p
  \end{pmatrix}$ and abuse notation by writing
\[
\sA y = \sum_{i=1}^{p+1} \sA^{i-1} y_{i},\mbox{ for }y\in\R^{p+1}.
\]
Then we can write
\[
\theta(x)= \sA
  \begin{pmatrix}
    1 \\
    x
  \end{pmatrix}=\sum_{i=1}^{p+1} (\sA N^i) u_i(x).
\]
Define $\Lambda^{i}=\sA N^i$ for $i\leq p+1$. It is sufficient to
prove that $\Lambda^i\geq0$ for $i\leq q$ and $\Lambda^i=0$ for
  $i>q$.

Let $i\leq q$ and write $N^i=(N_1^i,N_2^i)$ with $N_1^i\in\R$,
$N_2^i\in\R^p$. We consider two cases.

\emph{Case 1.} Suppose $N_1^i\not=0$. Then
$N_2^i/N_1^i\in\{\theta\geq0\}=:\mathcal{D}$. Indeed, for $j\in Q$
we have
\[
u_j(N_2^i/N_1^i)=\frac{1}{N_1^i}\Gamma_j N^i = \frac{1}{N_1^i}
  1_{\{i=j\}}.
\]
So $u_j(N_2^i/N_1^i)=0$ for $j\in Q\backslash\{i\}$, whence by
assumption $u_i(N_2^i/N_1^i)\geq0$, i.e.\ $N_1^i>0$ and
$N_2^i/N_1^i\in\mathcal{X}\subset\mathcal{D}$. It follows that
\[
0\leq\theta(N_2^i/N_1^i)= \frac{1}{N_1^i}\sA N^i=
\frac{1}{N_1^i}\Lambda^i,
\]
and thus $\Lambda^i\geq0$, as $N_1^i>0$.

\emph{Case 2.} Suppose $N_1^i=0$. Then we have
\begin{align*}\Lambda^i=
\sA
  N^i=\lim_{\varepsilon\downarrow0}\varepsilon \sA
  \begin{pmatrix}
    1 \\
    N_2^i/\varepsilon
  \end{pmatrix}=\lim_{\varepsilon\downarrow0}\varepsilon\theta(N_2^i/\varepsilon).
\end{align*}
So we have to prove $N_2^i/\varepsilon\in\mathcal{D}$ for
$\varepsilon$ small enough. For $j\in Q$  it holds that
\begin{align*}
1_{\{i=j\}}=\Gamma_j
  \begin{pmatrix}
    0 \\
    N_2^i
  \end{pmatrix}
=\gamma_j N_2^i.
\end{align*}
This gives
\[
u_j(N_2^i/\varepsilon)=\delta_j+1_{\{i=j\}}/\varepsilon\geq0,\mbox{
for }j\in Q,\varepsilon>0,
\]
since $\delta_j\geq0$ for $j\in Q$ by assumption. Hence
$N_2^i/\varepsilon\in\mathcal{X}\subset\mathcal{D}$ for
$\varepsilon>0$ and thus $\Lambda^i\geq0$.

We have just shown that $\Lambda^i\geq0$ for $i\leq q$. We now
show that $\Lambda^i=0$ for $i>q$. Let $i>q$ and take
$x_0\in\mathcal{X}$. Write $R=P\cup\{p+1\}$. Since $\rank\gamma=p$
and $\rank \gamma_{R\backslash\{i\}}=p-1$, there exists
$\xi\in\R^p$ such that $\gamma_j\xi=0$ for all $j\in R
\backslash\{i\}$ and $\gamma_i\xi\not=0$. Then for all $k\in\R$ we
have $u_j(x_0+k\xi)=u_j(x_0)$ for $j\in R\backslash\{i\}$, in
particular for all $j\in Q$. Hence $x_0+k\xi\in\mathcal{X}$ for
all $k\in\R$. Therefore,
\[
0\leq \theta(x_0+k\xi)=\sum_{j=1}^{p+1}\Lambda^ju_j(x_0+k\xi)=
\theta(x_0)+k\Lambda^i\gamma_i\xi,\mbox{ for all $k\in\R$}.
\]
Since $\gamma_i\xi\not=0$, it follows that $\Lambda^i=0$.

Now consider the general case without the restriction
$\delta_i\geq0$ for all $i\in Q$. Let $x_0\in\mathcal{X}$ and for
$i\in Q$ define
\[
\widetilde{u}_i:\R^p\rightarrow\R:x\mapsto \widetilde{\gamma}_i x
+\widetilde{\delta_i},
\]
by $\widetilde{u}_i(x)=u_i(x+x_0)$, $x\in\R^p$, i.e.\
$\widetilde{\delta}_i=\delta_i+\gamma_i x_0$ and
$\widetilde{\gamma}_i=\gamma_i$. Moreover, write
$\widetilde{\theta}(x)=\theta(x+x_0)$, $x\in\R^p$. Then
$\widetilde{\delta}_i=\widetilde{u}_i(0)=u_i(x_0)\geq0$, $i\in Q$.
Note that if $\widetilde{u}_j(x)=0$ for all $j\in
Q\backslash\{i\}$, then ${u}_j(x+x_0)=0$ for all $j\in
Q\backslash\{i\}$ and hence $\widetilde{u}_i(x)={u}_i(x+x_0)\geq0$
by assumption. Moreover,
\[
  \begin{pmatrix}
    \widetilde{\delta} & \widetilde{\gamma}
  \end{pmatrix}
=  \begin{pmatrix}
    \delta+\gamma x_0 & \gamma
  \end{pmatrix}
  =  \begin{pmatrix}
    \delta& \gamma
  \end{pmatrix}
  \begin{pmatrix}
    1 & 0 \\
    x_0 & \Id
  \end{pmatrix},
\]
which has full row-rank. Therefore we are in the previous
situation, for which we have proved the existence of positive
semi-definite $\Lambda^i\in\R^{p\times p}$ for $i\in Q\cup\{0\}$
such that
\[
{\theta}(x+x_0)=\widetilde{\theta}(x)=\Lambda^0+\sum_{i=1}^q\Lambda^i\widetilde{u}_i(x)=\Lambda^0+\sum_{i=1}^q\Lambda^i{u}_i(x+x_0).
\]
This gives the result.
\end{proof}
\begin{cor}\label{cor:contained}
Let $\mathcal{X}\subset\R^p$ be given by (\ref{eq:convexpolyX}).
Suppose $\mathcal{X}$ is contained in a polyhedron
$\mathcal{C}\subset\{\theta\geq0\}$ which meets the criteria of
Proposition~\ref{prop:gammafullrank}. Then (\ref{eq:thetapossum})
holds.
\end{cor}
\begin{proof}
Write $\mathcal{C}=\bigcap_{i=1}^r \{v_i\geq0\}$, for some affine
functions $v_i$, with $r\in\N$. By
Proposition~\ref{prop:gammafullrank} it holds that
\begin{align}\label{eq:thetasumvi}
\theta(x)=B^0+\sum_{i=1}^r  B^i v_i(x),
\end{align}
for some positive semi-definite $B^i$. By
Proposition~\ref{prop:sht} there exist $\lambda_{ij}\geq0$,
$c_i\geq0$ such that
\[
v_i=\sum_{j=1}^r \lambda_{ij}u_j+c_i.
\]
Plugging this in into~(\ref{eq:thetasumvi}) we get the result.
\end{proof}

\begin{example}
In the $2$-dimensional case, the polyhedrons which satisfy the
conditions of Proposition~\ref{prop:gammafullrank} are the
``triangles'' (including those with vertices and edges in
``infinity'', like $\{0\leq x_1\leq 1\}\cap\{x_2\geq0\}$ and
$\{x_1\geq0\}$). Thus by Corollary~\ref{cor:contained}, if
$\mathcal{X}$ is contained in a triangle that is a subset of
$\{\theta\geq0\}$, then (\ref{eq:thetapossum}) holds and
$\mathcal{X}$ can be transformed in canonical form (see the proof
of Proposition~\ref{prop:phipos}). However, this condition is
sufficient, but not necessary. For example let
\begin{align*}
\theta(x)&=
  \begin{pmatrix}
    x_1 & 1 \\
    1 & x_2
  \end{pmatrix}=\begin{pmatrix}
    0 & 1 \\
    1 & 0
  \end{pmatrix}+\begin{pmatrix}
    1 & 0 \\
    0 & 0
  \end{pmatrix}x_1+\begin{pmatrix}
    0 & 0 \\
    0 & 1
  \end{pmatrix}x_2,
\end{align*}
and take $u_1(x)=2x_1-x_2$, $u_2(x)=-\half x_1 +x_2$,
 $u_3(x)=-\frac{9}{4}+x_1+x_2$. Then $\{\theta\geq0\}=\{x\in\R^2_{\geq0}:x_2\geq\frac{1}{x_1}\}$ and
$\mathcal{X}=\bigcap_{i=1}^3\{u_i\geq0\}\subset\{\theta\geq0\}$,
but $\mathcal{X}$ is not contained in a triangle which is a subset
of $\{\theta\geq0\}$. Still we can write $\theta$ as an affine
transformation of the $u_i$'s with positive semi-definite
coefficients, namely
\[
\theta=\frac{1}{2}
  \begin{pmatrix}
    1 & 1 \\
    1 & 1
  \end{pmatrix}
+\frac{1}{9}
  \begin{pmatrix}
    4 & 2 \\
    2 & 1
  \end{pmatrix}u_1+\frac{1}{9}
  \begin{pmatrix}
    2 & 4 \\
    4 & 8
  \end{pmatrix}u_2+\frac{1}{9}
  \begin{pmatrix}
    2 & -2 \\
    -2 & 2
  \end{pmatrix}u_3.
\]
\end{example}

\subsection{Classical model} In this
subsection we revisit the classical model as introduced by Duffie
and Kan in \cite{dk96}. We assume $\theta$ is of the form
\begin{align}\label{eq:specialtheta}
\theta=\Sigma\,\diag(v)\Sigma^\top=\sum_{i=1}^p
\Sigma^i(\Sigma^i)^\top v_i,
\end{align}
with $\Sigma\in\R^{p\times p}$, $v:\R^p\rightarrow\R^p:x\mapsto
\beta x+\alpha$ for some $\beta\in\R^{p\times p}$,
$\alpha\in\R^p$. In addition we assume
$\mathcal{X}\subset\bigcap_{i=1}^p\{v_i\geq 0\}$ and
$\mathcal{X}^\circ\not=\emptyset$. Under
conditions~(\ref{al:cond1}) and (\ref{al:cond2}), the proof of
Theorem~\ref{th:constructsigma} constructs a square root $\sigma$
of $\theta$ on $\mathcal{X}$ such that (\ref{eq:SDE}) has a unique
strong solution. We show that the natural choice
\[
\sigma=\Sigma\,\diag(\sqrt{|v|})=
  \begin{pmatrix}
    \Sigma^1\sqrt{|v_1|} &\ldots &\Sigma^p\sqrt{|v_p|}  \end{pmatrix}
\]
also gives strong existence and uniqueness for (\ref{eq:SDE}).
This is not immediately clear in view of
Remark~\ref{rem:generator}.
\begin{prop}\label{prop:classicstrong}
Suppose $\theta$ is of the form (\ref{eq:specialtheta}) and assume
$\mathcal{X}^\circ\not=\emptyset$,
$\mathcal{X}\subset\bigcap_{i=1}^p\{v_i\geq 0\}$ and
conditions~(\ref{al:cond1}) and (\ref{al:cond2}) are met. Then
(\ref{eq:SDE}) is an affine SDE for
$\sigma=\Sigma\,\diag(\sqrt{|v|})$.
\end{prop}
\begin{proof}
Since conditions~(\ref{al:cond1}) and (\ref{al:cond2}) are
sufficient for invariance by Theorem~\ref{th:constructsigma}, it
suffices to prove existence and uniqueness of a strong solution.

Note that $X$ solves an affine SDE with
$\sigma(x)=\Sigma\,\diag(\sqrt{|v(x)|})$ if and only if $LX+\ell$
solves an affine SDE with
$\sigma(x)=L\Sigma\,\diag(\sqrt{|\wv(x)|})$, where
$\wv(x)=v(L^{-1}(x-\ell))$, for some non-singular matrix $L$ and
vector $\ell$. By Proposition~\ref{prop:LthetaLT} we can therefore
assume without loss of generality that $\theta$ is of the form
(\ref{eq:canonicaltheta}) and $\mathcal{X}$ of the form
(\ref{eq:canonicalX}).
Since all square roots of a positive semi-definite matrix are
related by an orthogonal transformation, we have
\begin{align}\label{eq:ortho}
    \Sigma\,\diag(\sqrt{v(x)})=\begin{pmatrix}
    \diag(\sqrt{x_M},0_N) & 0 \\
    0 & (\Psi(x_{M\cup N}))^{1/2}
  \end{pmatrix}O(x),
\end{align}
for $x\in\mathcal{X}$, with $O(x)$ an orthogonal matrix, possibly
depending on $x$. We show that there exists a matrix
$U\in\R^{(m+n)\times p}$ with orthonormal rows such that
\[
\Sigma_{M\cup N}
\,\diag(\sqrt{|v(x)|})=\diag(\sqrt{|x_M|},0_N)U,\mbox{ for all
}x\in\R^p.
\]
Let $x_0\in\mathcal{X}^\circ$ and define $U=O_{M\cup N}(x_0)$. We
have to show that
\begin{align}\label{eq:orthoU}
\Sigma_{ij}\sqrt{|v_j(x)|}=\sqrt{|x_i|}U_{ij},\mbox{ for all $i\in
M$, all $j$}.
\end{align}
For $i\in M$ we have
\[
\Sigma_{ij}\sqrt{v_j(x)}=\sqrt{x_i}O_{ij}(x),\mbox{ for all $j$
and all $x\in\mathcal{X}$}.
\]
If $\Sigma_{ij}\neq 0$, then
\[
v_j(x)=\Sigma_{ij}^{-2} x_i O_{ij}^2(x),\mbox{ for all
}x\in\mathcal{X},
\]
which yields $v_j(x)=c x_i$ for some $c\geq0$, by
Proposition~\ref{prop:uiscv}. If $c>0$, then $O_{ij}(x)$ is
constant, whence equal to $U_{ij}$ and (\ref{eq:orthoU}) follows.
If $c=0$ or $\Sigma_{ij}=0$, then $O_{ij}(x)=0$ for $x_i\neq 0$,
whence $U_{ij}=0$ and again (\ref{eq:orthoU}) holds.

Let $(X,W)$ be a weak solution to (\ref{eq:SDE}) on some filtered
probability space $\filpspace$. Then $UW$ is an
$(m+n)$-dimensional Brownian motion and it follows that $X_{M\cup
N}$ solves an SDE with diffusion part $\diag(\sqrt{|x_M|},0_N)$.
The strong existence and uniqueness for (\ref{eq:SDE}) follows
along the same lines as the proof of
Theorem~\ref{th:constructsigma}, as $\sigma(x)$ only depends on
$x_{M\cup N}$ by Proposition~\ref{prop:sht}.
\end{proof}
\medskip

In the case that $\{\theta>0\}\neq\emptyset$, we can strengthen
the result of Proposition~\ref{prop:phipos}.
\begin{prop}
Consider the situation of Proposition~\ref{prop:classicstrong}. If
$\{\theta>0\}\neq\emptyset$, then the solution to (\ref{eq:SDE})
can be obtained by a bijective affine transformation of an affine
diffusion with diagonal diffusion matrix and canonical state space
of the form $\R^m_{\geq0}\times\R^{p-m}$.
\end{prop}
\begin{proof}
As in the proof of Proposition~\ref{prop:classicstrong} we may
assume $\theta$ is of the form (\ref{eq:canonicaltheta}) and
$\mathcal{X}$ of the form (\ref{eq:canonicalX}). Since
$\{\theta>0\}\neq\emptyset$ we have $N=\emptyset$. Take
$x_0\in\mathcal{X}^\circ=\R^m_{>0}\times\R^{p-m}$ with $x_{0,i}=1$
for $i\in M$. By strict positive definiteness of $\Psi(x_{0,M})$
we can apply the linear transformation
\[
x\mapsto
  \begin{pmatrix}
    I_{MM} & 0 \\
    0 & (\Psi(x_{0,M}))^{-1/2}
  \end{pmatrix}x,
\]
so that we may assume without loss of generality that
$\Psi(x_{0,M})=\Id$, i.e.\ $\theta(x_0)=\Id$. Since $v_i(x_0)>0$
for all $i$, we can substitute $v_i/v_i(x_0)$ for $v_i$ and
$\sqrt{v_i(x_0)}\Sigma^i$ for $\Sigma^i$, which does not affect
$\sigma(x)$ and yields $v_i(x_0)=1$ for all $i$. Hence
\[
\Sigma\Sigma^\top=\theta(x_0)=\Id,
\]
in other words, $\Sigma$ is orthogonal. Thus we can write
\[
\Sigma\,\diag(v(x))=
  \begin{pmatrix}
    \diag(x_M) & 0 \\
    0 & \Psi(x_M)
  \end{pmatrix}\Sigma.
\]
By symmetry, $\Psi(x_M)$ can be diagonalized by an orthogonal
matrix $U(x_M)$ for all $x_M$. We show that $U(x_M)$ is constant.
The above display yields
\[
\Sigma_{P\backslash M}\,\diag(v(x))=\Psi(x_M)\Sigma_{P\backslash
M},
\]
so for all $x_M$ the eigenvectors of $\Psi(x_M)$ are in the span
of the columns of $\Sigma_{P\backslash M}$, since
$\rank\Sigma_{P\backslash M}=p-m\geq \rank\Psi(x_M)$. Hence the
eigenvectors do not depend on $x_M$, which implies that $U(x_M)$
is constant, say equal to an orthogonal matrix $O$. Thus
$O^\top\Psi(x_M)O=\diag(w(x_M))$, for some affine vector-valued
function $w$. Applying the orthogonal transformation
\[
x\mapsto
  \begin{pmatrix}
    I_{MM} & 0 \\
    0 & O^\top
  \end{pmatrix}x,
\]
we may assume $\Psi(x_M)$ is diagonal. Hence
\[
\Sigma\,\diag(v(x))=\begin{pmatrix}
    \diag(x_M) & 0 \\
    0 & \diag(w(x_M))
  \end{pmatrix}\Sigma,
\]
with $\Sigma$ orthogonal. It is easy to show that this yields
\[
\Sigma\,\diag(\sqrt{|v(x)|})=\begin{pmatrix}
    \diag(\sqrt{|x_M|}) & 0 \\
    0 & \diag(\sqrt{|w(x_M)|})
  \end{pmatrix}\Sigma.
\]
Since an orthogonal transformation of a Brownian motion is again a
Brownian motion, the $\Sigma$ on the right-hand side can be
absorbed in the underlying Brownian motion. The result follows.
\end{proof}
\medskip

Using the above we are able to give an alternative proof of the
existence and uniqueness results from \cite{dk96}, slightly
strengthening the statements made there, see
Remark~\ref{remark413} below. We first show how the well-known
condition for invariance of the open state space $\R^p_{>0}$
follows from Proposition~\ref{prop:nablaphi2}.
\begin{prop}\label{prop:feller}
Let $\theta(x)=\diag(x)$ and $\mathcal{X}=\R^p_{\geq0}$. Then
$\mathcal{X}^\circ$ is stochastically invariant if
\[
\forall i,\forall j\neq i: a_{ij}\geq0\mbox{ and } b_i\geq\half.
\]
\end{prop}
\begin{proof}
We apply Proposition~\ref{prop:nablaphi2} with
$\Phi(x)=\det\theta(x)=x_1 x_2\ldots x_p$. It holds that
\[
\nabla\Phi(x)\theta(x)=\Phi(x)
  \begin{pmatrix}
    1 & \ldots & 1
  \end{pmatrix},
\]
and $\sum_{i=1}^p (A^i)^i)=\begin{pmatrix}
    1 & \ldots & 1
  \end{pmatrix}^\top$, whence
\[
\nabla\Phi(x)(\mu(x)-\half\sum_{i=1}^p (A^i)^i)=\sum_i
\prod_{j\neq i} x_j (a_i x + b_i-\half).
\]
This is non-negative for all $x\in\mathcal{X}=\R^p_{\geq0}$ if
$a_{ij}\geq 0$ and $b_i\geq\half$ for all $i,j$. By applying a
measure transformation with density
$\mathcal{E}(\lambda^\top\sigma(X)\cdot W)$ for some
$\lambda\in\R^p$ (which yields a probability measure
by~\cite[Corollary~A.9]{part1}), we see that the sign of the
diagonal elements $a_{ii}$ is irrelevant for stochastic invariance
of $\mathcal{X}^\circ$.
\end{proof}
\begin{cor}\label{cor:polyinvar}
Suppose $\theta$ is of the form (\ref{eq:specialtheta}) and assume
$\mathcal{X}^\circ\not=\emptyset$,
$\mathcal{X}\subset\bigcap_{i=1}^p\{v_i\geq 0\}$ and
$\{\theta>0\}\not=\emptyset$. Then (\ref{eq:SDE}) is an affine SDE
for $\sigma=\Sigma\sqrt{|v|}$ if and only if, up to a
reparametrization of $\Sigma$ and $v$,
$\mathcal{X}=\bigcap_{i=1}^q\{v_i\geq
0\}=\{\theta\geq0\}=:\mathcal{D}$ for some $q\leq p$ and
\begin{align}
\forall i\leq q,\forall j\leq p&:\beta_i\Sigma^j=0\mbox{ or
}v_i=v_j,\label{al:wf1}\\
\forall i\leq q,\forall x\in\Di&:\beta_i(a x
+b)\geq0.\label{al:wf2}
\end{align}
Moreover, if we strengthen (\ref{al:wf2}) to
\begin{align}\label{eq:strongfeller}
\forall i\leq q,\forall x\in\Di:\beta_i(a x
+b)\geq\half\beta_i\Sigma\Sigma^\top\beta_i^\top,
\end{align}
then $\mathcal{D}^\circ=\bigcap_{i=1}^q\{v_i>0\}$ is
stochastically invariant.
\end{cor}
\begin{proof}
Note that $\{\theta>0\}\neq\emptyset$ implies $\Sigma$ is
non-singular. Hence $\theta\geq0$ if and only if
$\Sigma^{-1}\theta(\Sigma^{-1})^\top=\diag(v)\geq0$, so
\begin{align}\label{eq:Disv}
\mathcal{D}=\bigcap_{i=1}^p\{v_i\geq0\}.
\end{align}
Therefore, if $\mathcal{X}$ is stochastically invariant, then
$\mathcal{X}=\mathcal{D}$, since $\theta(x)$ is singular for
$x\in\partial\mathcal{X}$ by condition (\ref{al:cond1}), required
for invariance. For the first part of the proposition it remains
to show that conditions~(\ref{al:cond1}) and (\ref{al:cond2}) are
equivalent to (\ref{al:wf1}) and (\ref{al:wf2}). Permuting and
rescaling the elements in $v$ and column vectors in $\Sigma$ we
may assume $v_i=u_i$ for $i\leq q$, i.e.\
\[
\mathcal{X}=\bigcap_{i=1}^q\{v_i\geq0\}.
\]
Now condition~(\ref{al:cond1}) reads
\[
\forall i\in Q,j\in P:\beta_i\Sigma^j=0\mbox{ or }v_j=c v_i \mbox{
for some }c>0,
\]
by Proposition~\ref{prop:uiscv}. The constant $c$ can be taken
equal to $1$ by rescaling $v$ and $\Sigma$. This gives the
equivalence between (\ref{al:cond1}) and (\ref{al:cond2}), while
the equivalence between (\ref{al:wf1}) and (\ref{al:wf2}) is
immediate.

For the second part of the proposition we follow the proof of
Proposition~\ref{prop:LthetaLT} up to Step~4. If we take
\[
u_i=(\beta_i\Sigma\Sigma^\top\beta_i^\top)^{-1} v_i,\mbox{ for
$i\leq q$},
\]
then (\ref{eq:Bi}) gives
\[
c_i=B_i\gamma_i^\top=v_i(x_0)^{-1}(\beta_i\Sigma\Sigma^\top\beta_i^\top)^{-1}
\beta_i\Sigma\,\diag(v(x_0))\Sigma^\top\beta_i^\top=1,
\]
by choosing $x_0\in\mathcal{D}^\circ$ such that $v_i(x_0)=1$ for
$i\leq q$. Hence the affine transformation $Lx+\ell$ from
Proposition~\ref{prop:LthetaLT} satisfies
\[
(Lx+\ell)_i=u_i(x)=(\beta_i\Sigma\Sigma^\top\beta_i^\top)^{-1}
v_i(x),\mbox{ for $i\in Q$}.
\]
Let $X$ be an affine diffusion with SDE~(\ref{eq:SDE}). We have
that $Y:=(LX+\ell)_Q$ satisfies an affine SDE of the form
\[
\dd Y_t =\wmu(Y_t)\dd t+\diag(\sqrt{|Y_t|})\dd W_t,
\]
with state space $\R^q_{\geq0}$, where we write
\[
\wmu(x)=L\mu(L^{-1}(x-\ell))=\wa x+\wb,
\]
for some matrix $\wa$ with non-negative off-diagonal elements and
vector $\wb$ with non-negative components. Note that
$\mathcal{D}^\circ$ is invariant for $X$ if and only if
$\R^q_{>0}$ is invariant for $Y$. By
Proposition~\ref{prop:feller}, the latter holds if
$\wb_i\geq\half$ for all $i$, or equivalently, if
$\wmu_i(y)\geq\half$ for all $y\in\R^q_{\geq 0}$ with $y_i=0$.
Substituting $y=Lx+\ell$ gives
\[
\mbox{$\wmu_i(Lx+\ell)\geq\half$ for all $x\in\Di$}.
\]
Since $\wmu(x)=L\mu(L^{-1}(x-\ell))$ and
$L_i=(\beta_i\Sigma\Sigma^\top\beta_i^\top)^{-1}\beta_i$, the
result follows.
\end{proof}
\begin{remark}\label{remark413}
In \cite{dk96} it is assumed that $q=p$. Moreover, the strong
existence and uniqueness is only proved under (\ref{al:wf1}) and
(\ref{eq:strongfeller}), whereas the case that the process $X$
might hit the boundary $\partial\mathcal{X}$ is not treated. Note
also that in \cite{dk96} the inequality in (\ref{eq:strongfeller})
is strict.
\end{remark}
\section{Quadratic state space}\label{sec:quadratic}
In this section we consider affine diffusions where the boundary
of the state space $\mathcal{X}$ is quadratic instead of linear.
Let us be given a quadratic function
\begin{align}\label{eq:quadraticPhi}
\Phi(x)=x^\top A x + b^\top x+ c,
\end{align}
for some symmetric non-zero $A\in\R^{p\times p}$, $b\in\R^p$,
$c\in\R$.  We take $\mathcal{X}=\{\theta\geq0\}$ where $\theta$ is
given by (\ref{eq:affinetheta}) (thus $\mathcal{X}$ is convex) and
we assume $\mathcal{X}^\circ$ is a non-empty connected component
(maximal connected subset) of $\{\Phi>0\}$ or $\{\Phi<0\}$. Note
that then automatically the boundary of the state space is
quadratic, i.e.\ $\partial\mathcal{X}\subset \{\Phi=0\}$.
 By the following proposition, there are only three types possible for $\Phi$.
\begin{prop}\label{prop:twophis}
Let $\mathcal{X}\subset\R^p$ be convex and assume
$\mathcal{X}^\circ$ is a non-empty connected component of
$\{\Phi>0\}$ or $\{\Phi<0\}$, with $\Phi$ given by
(\ref{eq:quadraticPhi}). Then there exists an affine
transformation such that either $\Phi(x)=x_1-\sum_{i=2}^q x_i^2$,
$\Phi(x)=\sum_{i=1}^q x_i^2 + d$ or $\Phi(x)=x_1^2-\sum_{i=2}^q
x_i^2+d$, for some $1\leq q\leq p$, $d\in\R$.
\end{prop}
\begin{proof}
Since $A$ is symmetric, it is diagonalizable by an orthogonal
matrix. By further scaling one can take the diagonal elements
equal to $-1$, $0$ or $1$. Using the equality
\[
x^\top x + b^\top x = (x^\top + \half b^\top)(x+\half b)-\quart
b^\top b,
\]
we can apply an affine transformation such that for some disjoint
$Q,Q'\subset P\backslash\{1\}$ the quadratic function $\Phi$ is of
the form
\begin{align}\label{eq:absurdPhi2}
\Phi(x)=x_1-\sum_{i\in Q} x_i^2 +\sum_{i\in Q' } x_i^2,
\end{align}
or
\begin{align}\label{eq:absurdPhi}
\Phi(x)=x_1^2-\sum_{i\in Q} x_i^2 +\sum_{i\in Q' } x_i^2+d,
\end{align}
for some $d\in\R$.  If $\Phi$ is of the form
(\ref{eq:absurdPhi2}), then $\mathcal{X}$ is of the form
\[
\mathcal{X}=\{x\in\R^p:x_1\geq \sum_{i\in Q} x_i^2 -\sum_{i\in Q'
} x_i^2\},
\]
possibly after replacing $x_1$ by $-x_1$ and interchanging $Q$ and
$Q'$. Convexity of $\mathcal{X}$ yields that the Hessian of
$\sum_{i\in Q} x_i^2 -\sum_{i\in Q' } x_i^2$ is positive
semi-definite, which implies that $Q'=\emptyset$. Permuting
coordinates gives $Q=\{2,\ldots,q\}$ with $q=\# Q+1$.

Now assume $\Phi$ is of the form (\ref{eq:absurdPhi}). We have to
show that either $Q'=\emptyset$ or $\# Q\leq 1$. Define the
function $f$ by $f(x_{Q\cup Q'})=\sum_{i\in Q} x_i^2 -\sum_{i\in
Q' } x_i^2-d$. There are two possible forms $\mathcal{X}$ can
assume, namely
\begin{align*}
\mathcal{X}&=\{x\in\R^p:x_1\geq f(x_{Q\cup Q'})^{1/2},x_{Q\cup
Q'}\in K\}\\
\mbox{or }\mathcal{X}&=\{x\in\R^p:|x_1|\leq f(x_{Q\cup
Q'})^{1/2},x_{Q\cup Q'}\in K\},
\end{align*}
where $K$ is convex with $K^\circ$ a non-empty connected component
of $\{f\geq 0\}$. In the first case we have that the Hessian of
$f(x_{Q\cup Q'})^{1/2}$ is positive semi-definite, while in the
second case the Hessian is negative semi-definite. Now suppose
$Q'\neq\emptyset$. We show that in that case $\#Q \leq 1$. Let
$x_{Q\cup Q'}\in K^\circ$. For $i\in Q'$ we have
\[
  \frac{\partial^2}{\partial x_i^2} f(x_{Q\cup Q'})^{1/2} = -f(x_{Q\cup
  Q'})^{-1/2}-x_i^2 f(x_{Q\cup Q'})^{-3/2}<0,
\]
hence the Hessian is negative semi-definite. For $i\in Q$ we have
\begin{align*}
\frac{\partial^2}{\partial x_i^2} f(x_{Q\cup Q'})^{1/2} &=
f(x_{Q\cup
  Q'})^{-1/2}-x_i^2 f(x_{Q\cup Q'})^{-3/2},
\end{align*}
which therefore also has to be negative. Now if $i,j\in Q$ and
$i\neq j$, then
\[
\frac{\partial^2}{\partial x_i\partial x_j} f(x_{Q\cup Q'})^{1/2}
= -f(x_{Q\cup
  Q'})^{-3/2}x_ix_j,
\]
from which we deduce that
\begin{align*}
&\det
  \begin{pmatrix}
    \frac{\partial^2}{\partial x_i^2} & \frac{\partial^2}{\partial x_i\partial x_j} \\
    \frac{\partial^2}{\partial x_j\partial x_i} & \frac{\partial^2}{\partial x_j^2}
  \end{pmatrix} f(x_{Q\cup Q'})^{-1/2}= \\&\qquad= f(x_{Q\cup
  Q'})^{-1/2}\left(f(x_{Q\cup
  Q'})^{-1/2}-(x_i^2+x_j^2) f(x_{Q\cup Q'})^{-3/2}\right)<0.
\end{align*}
This contradicts the negative semi-definiteness of the Hessian.
Thus it holds that $\# Q\leq 1$.
%
\end{proof}
\medskip

If (\ref{eq:SDE}) is an affine SDE with drift $\mu$, diffusion
matrix $\theta$ and state space $\mathcal{X}=\{\theta\geq0\}$ with
non-empty quadratic boundary $\partial\mathcal{X}\subset
\{\Phi=0\}$, then stochastic invariance of $\mathcal{X}$ yields
\begin{align}\label{eq:gammaxthetaxis0}
\nabla\Phi(x) \theta(x) = 0,\mbox{ for all
}x\in\partial\mathcal{X}.
\end{align}
by Proposition~\ref{prop:necinvar} and the remark preceding
Proposition~\ref{prop:invargeneral}. This excludes that $\Phi$ is
of the form $\Phi(x)=\sum_{i=1}^q x_i^2 + d$ or
$\Phi(x)=x_1^2-\sum_{i=2}^q x_i^2+d$ with $d\neq 0$. Indeed,
suppose $\Phi(x)=\sum_{i=1}^p x_i^2 + d$ with $d<0$ (for
simplicity we take $q=p$), then (\ref{eq:gammaxthetaxis0}) reads
\[
x^\top \theta(x) = 0, \mbox{ for all $x$ such that $\sum_{i=1}^p
x_i^2 =-d$}.
\]
By Lemma~\ref{lem:psiiscphi} below and the observation that the
degree of $x^\top\theta(x)$ does not exceed the degree of
$\Phi(x)$, it follows that
\[
x^\top\theta(x)= (\sum_{i=1}^p x_i^2 +d)c^\top,\,\mbox{ for all
}x\in\R^p,
\]
for some constant vector $c\in\R^p$. However, since $x^\top
\theta(x)$ has no constant terms, this yields $x^\top \theta(x) =
0$ for all $x\in\R^p$. Lemma~\ref{lem:thetais0} below gives that
$\theta(x)=0$ for all $x$, which contradicts the assumption that
$\mathcal{X}=\{\theta\geq0\}$ and
$\partial\mathcal{X}\neq\emptyset$. Likewise one can show that
$\Phi(x)=x_1^2-\sum_{i=2}^q x_i^2+d$ with $d\neq 0$ is impossible.
\begin{lemma}\label{lem:psiiscphi}
Let $\Phi,\Psi:\R^p\rightarrow\R$ be multivariate polynomials and
$O\subset\R^p$ a non-empty open set such that for $x\in O$ we have
$\Psi(x)=0$ whenever $\Phi(x)=0$. Suppose in addition that for
$x\in O$, $\Phi(x)$ is of the form
\[
\Phi(x)=\prod_{j=1}^d (x_1-r_j(x_{P\backslash\{1\}})),
\]
for some $d\in\N$ and some $r_j(x_{P\backslash\{1\}})\in\R$ which
are mutually different for all $x\in O$. Then
$\Psi(x)=C(x)\Phi(x)$ for some multivariate polynomial $C$.
\end{lemma}
\begin{proof}
Viewing $x_1\mapsto \Phi(x)$ as a univariate polynomial in $x_1$
with distinct roots $r_j(x_{P\backslash\{1\}})$ for fixed
$x_{P\backslash\{1\}}$, it follows by the factor theorem for
polynomials that
\[
\Psi(x)=C(x)\Phi(x),\,\mbox{ for }x\in O,
\]
for some function $C$ that is a polynomial in $x_1$. Since $\Psi$
is a multivariate polynomial that only depends on $x_1$ by the
term $C(x) x_1^d$, it follows that $C$ is a multivariate
polynomial. The equality in the above display can be extended to
$\R^p$ by applying a uniqueness theorem for holomorphic functions,
see \cite[p.\ \!226]{rudin}. This yields the result.
\end{proof}
\begin{lemma}\label{lem:thetais0}
Let $\theta:\R^p\rightarrow\R^{p\times p}:x\mapsto
A^0+\sum_{k=1}^p A^k x_k$ with $A^k\in\R^{p\times p}$ symmetric
and assume $x^\top\theta(x)=0$ for all $x$. Then $\theta(x)=0$ for
all $x$.
\end{lemma}
\begin{proof}
It is clear that $A^0=0$. We show that $A^k=0$ for $k>0$. It holds
that
\begin{align*}
0=\big(x^\top \sum_{k=1}^{p} A^k
x_k\big)_j=\sum_{k=1}^{p}\sum_{i=1}^{p}  x_i x_k
A^k_{ij}&=\sum_{1\leq i<k\leq p} x_i x_k
(A^k_{ij}+A^i_{kj})\\&\qquad+\sum_{i=1}^{p} x_i^2 A^i_{ij},
\end{align*}
for all $j$. Hence for all $i,j$ we have $A^i_{ij}=0$ and
$A^k_{ij}=-A^i_{kj}$ for $k\neq i$. Since $A^k$ is symmetric, the
latter gives $A^k_{ij}=-A^i_{kj}=-A^i_{jk}$. So if we permute the
indices $i,j,k$ by the cycle $(i\mapsto j, j\mapsto k, k\mapsto
i)$, then $A^k_{ij}$ gets a minus sign. Permuting the indices
repeatedly we obtain
\[
A^k_{ij}=-A^i_{jk}=A^j_{ki}=-A^k_{ij},
\]
which implies $A^k_{ij}=0$ for all $i,j$ and $k\neq i$. Hence
$A^k=0$ for all $k$, as we have already shown that $A^i_{ij}=0$
for all $i,j$.
\end{proof}
\medskip

Thus in order to characterize all affine diffusions with quadratic
state space, there are two cases to consider, namely
$\Phi(x)=x_1-\sum_{i=2}^q x_i^2$ and $\Phi(x)=x_1^2-\sum_{i=2}^q
x_i^2$. In the next subsections we characterize for these two
forms of $\Phi$ all possible $\theta$ which can act as a diffusion
matrix of an affine SDE, i.e.\ which $\theta$ satisfy
(\ref{eq:gammaxthetaxis0}). Moreover, we are able to construct a
square root $\sigma$ of $\theta$ such that (\ref{eq:SDE}) is an
affine SDE with quadratic state space $\mathcal{X}$, generalizing
the 2-dimensional setting as treated in \cite[Section~12]{dfs03}
and \cite{gs06}. In particular we show existence and uniqueness of
a strong solution.

\subsection{Parabolic state space} Assume $\Phi$  is of the form
$ \Phi(x)=x_1-\sum_{i=2}^q x_i^2, $ with $1<q\leq p$. The state
space $\mathcal{X}$ then necessarily equals
$\mathcal{X}=\{\Phi\geq 0\}$. For $x\in\R^p$ we write
$x=(x_1,y,z)\in\R^{1}\times\R^{q-1}\times\R^{p-q}$ and we define
affine matrix-valued functions $\zeta$ and $\eta$ by
\[
\zeta(x)=\begin{pmatrix}
    4x_1 & 2y^\top \\
    2y & \Id
  \end{pmatrix},\quad \eta(x)=\begin{pmatrix}
    0 & 0 & \ldots & 0 \\
    T_{12}(y) & T_{13}(y) & \ldots & T_{q-2,q-1}(y)
  \end{pmatrix},
\]
with $T_{ij}:\R^{q-1}\rightarrow\R^{q-1}$ for $1\leq i<j<q$ given
by $T_{ij}(y)_i=y_j$, $T_{ij}(y)_j=-y_i$, $T_{ij}(y)_k=0$ for
$k\neq i,j$. Moreover, we write $Q=\{2,\ldots,q\}$. Now,
condition~(\ref{eq:gammaxthetaxis0}) reads
\begin{align}\label{eq:gammaxthetaxis1}
  \begin{pmatrix}
    1 & -2y^\top & 0
  \end{pmatrix}\theta(x)=0\mbox{ for all $x\in\R^p$ such that }x_1=y^\top y.
\end{align}
We use the following lemmas to characterize those $\theta$ that
satisfy this condition in Proposition~\ref{prop:quadrtheta}.
\begin{lemma}\label{lem:basis}
Consider the linear space
\begin{align}\label{eq:linspace}
\mathcal{L}=\left\{ a:\R^p\rightarrow\R^q\mbox{ affine }\,|\,
  \begin{pmatrix}
    1 & -2y^\top
  \end{pmatrix}a(x)=0\mbox{ for all $x$ with }x_1=y^\top y
\right\}.
\end{align}
Then a basis for $\mathcal{L}$ is formed by the columns of $\zeta$
and $\eta$.
\end{lemma}
\begin{proof}
Clearly these columns are linearly independent elements of
$\mathcal{L}$. To prove that they span $\mathcal{L}$ we use a
dimension argument. Let $\mbox{Aff}(\R^p,\R^q)$ denote the space
of affine functions from $\R^p$ to $\R^q$ and let
$\mbox{Quadr}(\R^p,\R)/(x_1-y^\top y)$ be the space of quadratic
functions from $\R^p$ to $\R$, modulo $x_1-y^\top y$ (that is, $p$
and $q$ are equivalent if $p(x)-q(x)= c(x_1-y^\top y)$ for some
constant $c$). Consider the linear operator
\[
L:\mbox{Aff}(\R^p,\R^q)\rightarrow
\mbox{Quadr}(\R^p,\R)/(x_1-y^\top y): a(x)\mapsto \begin{pmatrix}
    1 & -2y^\top
  \end{pmatrix} a(x),
\]
and note that $\mathcal{L}=\ker L$, in view of
Lemma~\ref{lem:psiiscphi}. By the dimension theorem for linear
operators, we have
\[
\dim\mbox{Aff}(\R^p,\R^q)= \dim\ker L+\dim\mbox{im}\, L.
\]
It holds that $\dim\mbox{Aff}(\R^p,\R^q)=pq+q$. Since $x_1\equiv
y^\top y$, a basis for $\mbox{im}\, L$ is given by
\[
\{1,x_2,\ldots,x_p\}\cup\{x_ix_j:2\leq i \leq q, 1\leq j\leq p\},
\]
whence
\[
\dim\mbox{im}\, L=p+(q-1)p-
  \begin{pmatrix}
    q-1 \\
    2
  \end{pmatrix}=pq-
  \begin{pmatrix}
    q-1 \\
    2
  \end{pmatrix}.
\]
It follows that $\dim\ker L=q+\half(q-1)(q-2)$, which is the
number of columns in $\zeta$ and $\eta$. Thus the columns span the
kernel of $L$.
\end{proof}
\begin{lemma}\label{lem:czeta}
Let $\mathcal{L}$ be defined by (\ref{eq:linspace}) and suppose
$M:\R^p\rightarrow\R^{q\times q}$ is affine and $M(x)$ is
symmetric for all $x\in\R^p$. If the columns of $M$ are in
$\mathcal{L}$, then $M=c\zeta$ for some $c\in\R$.
\end{lemma}
\begin{proof}
By Lemma~\ref{lem:basis} there exist matrices $A$ and $B$ such
that
\[
M(x)=\zeta(x) A + \eta(x) B.
\]
Write  $T(y)=(T_{ij}(y))_{1\leq i<j<q}$ and $B=
  \begin{pmatrix}
    B^1 & \wB
  \end{pmatrix}$. Then the above display reads
\[
M(x)=
  \begin{pmatrix}
    4x_1 A_{11} +2 y^\top A_{Q1} & 4 x_1 A_{1Q} + 2y^\top A_{QQ} \\
    2y A_{11} + A_{Q1}+T(y) B^1 & 2y A_{1Q} + A_{QQ}+T(y) \wB
  \end{pmatrix}
\]
Since $M(x)$ is symmetric it immediately follows that $A_{1Q}=0$
and $A_{Q1}=0$. Define $N(x)=M(x)-A_{11}\zeta(x)$. Then $N$ is
also symmetric. We have
\[
N(x)=\begin{pmatrix}
    0  &  2y^\top C \\
    T(y) B^1  &  C+T(y)\wB
  \end{pmatrix},
\]
with $C=A_{QQ}-A_{11}\Id$. This yields $C=C^\top$ and $T(y)B^1=2C
y$. Since $y^\top T(y)=0$, the latter implies $y^\top C y =0$,
whence $C=0$, as $C$ is diagonalizable by an orthogonal matrix.
Thus $A_{QQ}=A_{11}\Id$ and it remains to show that $\wB=0$.

It holds that $T(y)\wB$ is symmetric and $y^\top T(y)\wB=0$.
Lemma~\ref{lem:thetais0} yields $T(y)\wB=0$,
whence $\wB=0$ by linear independence of the columns of $T(y)$, as
we needed to prove.
\end{proof}

\begin{prop}\label{prop:quadrtheta}
If (\ref{eq:gammaxthetaxis1}) holds, then necessarily $\theta$ is
of the form
\begin{align}\label{eq:quadrtheta2}
\theta(x)=
 \begin{pmatrix}
    c\zeta(x) & A(x) \\
    A(x)^\top & B(x)
  \end{pmatrix},
\end{align}
for some $c\geq0$, with $A(x)=\zeta(x) A_1+\eta(x) A_2$ for some
matrices $A_1,A_2$ and $B:\R^p\rightarrow\R^{(p-q)\times (p-q)}$
affine and symmetric. Moreover, if $c=1$ and $A_1=0$, then it
holds that
\begin{align}\label{eq:Bgeqbla}
B(x)- A_2^\top\eta(x)^\top\eta(x)A_2\geq0,
\end{align}
for all $x\in\mathcal{X}=\{\theta\geq0\}$.
\end{prop}
\begin{proof}
The first part follows from Lemma~\ref{lem:basis} and
Lemma~\ref{lem:czeta}. It remains to show (\ref{eq:Bgeqbla}).

Suppose $c=1$ and $A_1=0$. By positive semi-definiteness of
$\theta$, we have
\begin{align*}
  0&\leq\begin{pmatrix}
    v^\top & w^\top
  \end{pmatrix}\theta(x)
  \begin{pmatrix}
    v \\
    w
  \end{pmatrix}\\&=v^\top\zeta(x) v +2v^\top\eta(x)
  A_2 w+w^\top
  B(x) w,
\end{align*}
for all $v\in\R^q$, $w\in\R^{p-q}$, $x\in\mathcal{X}$.  Fix
$w\in\R^{p-q}$, $x\in\mathcal{X}$ arbitrarily and take
$v=-\eta(x)A_2w$. Noting that $\zeta(x)\eta(x)=\eta(x)$ for all
$x\in\R^p$, the above display then reads
\[
w^\top
  B(x) w -
  w^\top A_2^\top\eta(x)^\top\eta(x)A_2 w \geq0,
\]
which proves (\ref{eq:Bgeqbla}).
\end{proof}

\medskip

To show the existence of an affine diffusion with parabolic state
space in Theorem~\ref{th:quadr}, we need the following result. Its
proof is based on a modification of a result by Yamada and
Watanabe \cite[Theorem~1]{YamadaI}.
\begin{prop}\label{prop:strongquadr}
There exists a unique strong solution to the SDE
\[
\dd
  \begin{pmatrix}
    X_{1,t} \\
    Y_{t}
  \end{pmatrix}=
  \begin{pmatrix}
    a_{11} X_{1,t} + a_{1Q} Y_t + b_1 \\
    a_{QQ} Y_t + b_Q
  \end{pmatrix}\dd t +
  \begin{pmatrix}
    2\sqrt{|X_{1,t}-Y_t^\top Y_t|} & 2Y_t^\top \\
    0 & \Id
  \end{pmatrix}\dd W_t.
\]
\end{prop}
\begin{proof}
By continuity of the coefficients and satisfaction of the linear
growth condition~(\ref{eq:lingrow}), there exists a weak solution
$((X_1,Y),W)$ on some filtered probability space $\filpspace$
carrying a Brownian motion $W$. For strong existence and
uniqueness it suffices to prove pathwise uniqueness. Therefore,
assume $((\wX_1,\wY),W)$ is another solution on the same
probability space. We see that the equation for $Y$ does not
contain $X_1$ and in fact it is an SDE which has a unique strong
solution, whence $Y_t=\wY_t$ a.s.\ for all $t\geq0$. Write
$Z=X_1-Y^\top Y$ and $\wZ=\wX_1-\wY^\top \wY$. Then it follows
that
\[
\dd (Z_t-\wZ_t)=a_{11} (Z_{1,t}-\wZ_{1,t})\dd t
+2\left(\sqrt{|Z_t|}-\sqrt{|\wZ_t|}\right)\dd W_{1,t}.
\]
Arguing as in the proof of \cite[Proposition~5.2.13]{Karshr} (for
instance), we deduce that $Z_t=\wZ_t$ a.s.\ for all $t\geq0$. Thus
$X_{1,t}=Z_t+Y_t^\top Y_t=\wZ_t+\wY_t^\top \wY_t=\wX_{1,t}$ a.s.\
for all $t\geq0$.
\end{proof}

\begin{theorem}\label{th:quadr}
Let $\Phi(x)=x_1-y^\top y$ and suppose $\{\theta\geq
0\}=\{\Phi\geq 0\}$. Then there exists an affine SDE with drift
$\mu$, diffusion matrix $\theta$ and state space
$\mathcal{X}=\{\theta\geq0\}$ if and only if
\begin{align}
  \nabla\Phi(x) \theta(x)&=0\label{con:invarPhi2},\\
\nabla\Phi(x) \mu(x)&\geq \tr(\theta_{QQ}(x))\label{con:invarPhi1}
\end{align}
for all $x\in\partial\mathcal{X}=\{\Phi=0\}$.
\end{theorem}
\begin{proof}
Conditions~(\ref{con:invarPsi2}) and (\ref{con:invarPsi1}) reduce
to (\ref{con:invarPhi2}) and (\ref{con:invarPhi1}). The ``only
if''-part follows, as these boundary conditions are necessary for
stochastic invariance of $\mathcal{X}$ by
Proposition~\ref{prop:invargeneral}.

Suppose the boundary conditions (\ref{con:invarPhi2}) and
(\ref{con:invarPhi1}) hold. Then by
Proposition~\ref{prop:quadrtheta} it holds that $\theta$ is of the
form (\ref{eq:quadrtheta2}) for some $c\geq0$. If $c=0$, then
necessarily $\mathcal{X}=\{B\geq0\}$ and we take
\[
\sigma(x)=
  \begin{pmatrix}
    0 & 0 \\
    0 & |B(x)|^{1/2}
  \end{pmatrix}.
\]
It is easy to see that (\ref{eq:SDE}) admits a unique strong
solution and that $\mathcal{X}$ is invariant. If $c>0$, then we
apply the linear transformation $x_1\mapsto c^{-1} x_1$, $y\mapsto
c^{-1/2}y$, $z\mapsto z-c^{-1}A_1(x_1,y)$, so that we may assume
$c=1$ and $A_1=0$. Then (\ref{eq:Bgeqbla}) holds and as a square
root of $\theta$ on $\mathcal{X}$ we take
\[
\sigma(x)=
  \begin{pmatrix}
    \xi(x) & 0 \\
    A_2^\top\eta(x)^\top & \rho(x)
  \end{pmatrix},
\]
with $\xi$ and $\rho$ defined by
\begin{align*}
\xi(x)&=\begin{pmatrix}
    2\sqrt{|x_1-y^\top y|} & 2y^\top \\
    0 & \Id
  \end{pmatrix},\\
  \rho&=|B - A_2^\top\eta^\top\eta A_2|^{1/2}.
\end{align*}
To see that $\sigma(x)\sigma(x)^\top=\theta(x)$ for
$x\in\mathcal{X}$, note that $\xi(x)\eta(x)=\eta(x)$ for all
$x\in\mathcal{X}$. Proposition~\ref{prop:invargeneral} gives that
$\mathcal{X}$ is invariant for (\ref{eq:SDE}). It remains to prove
existence and uniqueness of a strong solution.

The boundary condition (\ref{con:invarPhi1}) reads
\[
\mu_1(x)-2y^\top\mu_Q(y)\geq q-1,\mbox{ for all $x$ such that
$x_1=y^\top y$},
\]
so necessarily $\mu_{\{1\}\cup Q}$ admits the form
\begin{align}\label{eq:mu1cupQ}
\mu_{\{1\}\cup Q}(x)=
  \begin{pmatrix}
    a_{11} & a_{1Q} & 0 \\
    0 & a_{QQ} & 0
  \end{pmatrix}
  \begin{pmatrix}
    x_1 \\
    y \\
    z
  \end{pmatrix}+b_{\{1\}\cup Q}.
\end{align}
Proposition~\ref{prop:strongquadr} gives the result.
\end{proof}
\medskip

For stochastic invariance of a parabolic state space
$\mathcal{X}$, we now give sufficient conditions on the diffusion
matrix and sufficient and necessary conditions on the drift,
analogous to the admissibility conditions for polyhedral state
spaces in canonical form.
\begin{prop}
Condition~(\ref{eq:Bgeqbla}) holds if $B$ admits the form
\[
B(x)=(q-2)x_1 A_2^\top A_2 .
\]
\end{prop}
\begin{proof}
One can show using Cauchy-Schwarz that
\[
\eta(x)_i \eta(x)_i^\top \Id \geq \eta(x)_i^\top \eta(x)_i,\mbox{
for all $x$ and $i$}.
\]
Hence
\[
\left(\sum_i \eta(x)_i \eta(x)_i^\top\right) \Id \geq \sum_i
\eta(x)_i^\top \eta(x)_i.
\]
The right-hand side equals $\eta(x)^\top\eta(x)$, while the
left-hand side is equal to $(q-2)\left(\sum_i y_i^2\right) \Id$,
which is smaller than $(q-2)x_1\Id$ for $x\in\mathcal{X}$. Hence
\[
(q-2) x_1 A_2^\top A_2 - A_2^\top\eta(x)^\top\eta(x) A_2=
A_2^\top((q-2)x_1\Id-\eta(x)^\top\eta(x)) A_2\geq0,
\]
for $x\in\mathcal{X}$, which yields the result.
\end{proof}
\begin{prop}\label{prop:admissiblequad}
Suppose $X$ is an affine diffusion with state space
$\mathcal{X}=\{\theta\geq0\}=\{x\in\R^p:x_1\geq y^\top y\}$ and
$\theta_{QQ}\neq 0$. Then $X$ is a linear transformation of an
affine diffusion with the same state space $\mathcal{X}$,
diffusion matrix $\theta$ of the form (\ref{eq:quadrtheta2}) with
$c=1$ and with drift $\mu(x)=a x + b$ satisfying
(\ref{eq:mu1cupQ}) as well as
\begin{align}
\label{eq:aQdiag} a_{11}\Id-2a_{QQ}&=\diag(d_Q)\\ d_{Q_1} &>0\\
 d_{Q_2} &=0\\ a_{1 Q_2}&=2b_{Q_2}\\
\label{eq:b1geqbla} b_1&\geq q-1+\sum_{i\in Q_1} \quart d_{i}^{-1}
(a_{1i}-2b_{i})^2,
\end{align}
for some vector $d$ and some disjoint $Q_1$ and $Q_2$ with
$Q=Q_1\cup Q_2$.
\end{prop}
\begin{proof}
As in the proof of Theorem~\ref{th:quadr} we can assume $\theta$
is of the form (\ref{eq:quadrtheta2}) with $c=1$. Then
$\tr(\theta_{QQ}(x))=q-1$ for all $x$, so the boundary
condition~(\ref{con:invarPhi1}) for the drift reads
\[
y^\top(a_{11}\Id-2a_{QQ})y+(a_{1Q}-2b_Q^\top)y+b_1-q+1\geq0,\mbox{
for all }y\in\R^{p-1}.
\]
For this it is necessary that $M:=a_{11}\Id-2a_{QQ}$ is positive
semi-definite. Moreover, if $y$ is in the kernel of $M$, then $y$
should also be in the kernel of $a_{1Q}-2b_Q^\top$. We can
diagonalize $M$ by an orthogonal matrix $O$, so $D=O M O^\top$ is
diagonal with positive diagonal elements $d_i$ for $i\in Q_1$ and
$d_i=0$ for $i\in Q_2=Q\backslash Q_1$, for some $Q_1\subset Q$.
Applying the orthogonal transformation $y\mapsto O y$, the above
condition becomes
\[
\sum_{i\in Q_1} d_i x_i^2+\sum_{i\in
Q_1}(a_{1i}-2b_i)x_i+b_1-q+1\geq0,\mbox{ for all }x_{Q_1}.
\]
We can write the left-hand side as
\[
\sum_{i\in Q_1} d_i (x_i + \half d_i^{-1}(a_{1i}-2b_i))^2 -\quart
\sum_{i\in Q_1}d_i^{-1}(a_{1i}-2b_i)^2 +b_1-q+1,
\]
which is non-negative for all $x_{Q_1}$ if and only if
\[
-\quart \sum_{i\in Q_1}d_i^{-1}(a_{1i}-2b_i)^2 +b_1-q+1\geq 0.
\]
This yields the result.
\end{proof}
\begin{remark}
Note that (\ref{eq:aQdiag}) implies that $a_{QQ}$ is diagonal.
Hence the coordinates of $X_Q$ are mutually independent.
\end{remark}
\begin{prop}\label{prop:quadrinvar}
Consider the situation of Proposition~\ref{prop:admissiblequad}.
If we strengthen condition~(\ref{eq:b1geqbla}) to
\begin{align}\label{eq:sterker}\tag{\ref*{eq:b1geqbla}$^\prime$}
b_1\geq q+1+\sum_{i\in Q_1} \quart d_{i}^{-1} (a_{1i}-2b_{i})^2,
\end{align}
then $\mathcal{X}^\circ$ becomes invariant.
\end{prop}
\begin{proof}
It suffices to verify condition~(\ref{al:phiv}) of
Proposition~\ref{prop:nablaphi2}. Recall that $\Phi(x)=x_1-y^\top
y$ and $\theta(x)=A^0+\sum_{i=1}^p A^i x_i$ is of the form
(\ref{eq:quadrtheta2}) with $c=1$. It follows that
$\nabla\Phi(x)(\mu(x)-\half\sum_{i=1}^p
  (A^i)^i)$ equals
\[
a_{11}x_1+a_{1Q}y-2y^\top a_{QQ}y-2y^\top b_Q+b_1-q-1.
\]
If in addition to conditions~(\ref{eq:aQdiag}) -
(\ref{eq:sterker}) also $a_{11}\geq0$ is imposed, then for
$x\in\mathcal{X}$ the above display is bounded from below by
\[
y^\top(a_{11}\Id-2a_{QQ})y+(a_{1Q}-2b_Q^\top)y+b_1-q-1.
\]
This is non-negative for all $y\in\R^{q-1}$ under the imposed
assumptions, similar as in the proof of
Proposition~\ref{prop:admissiblequad}, which yields
(\ref{al:phiv}). The non-negativity of $a_{11}$ can be dispensed
with, as shown as follows.

By applying a measure transformation with density
$\mathcal{E}(\lambda^\top\sigma(X)\cdot W)$ for some
$\lambda\in\R^p$ with $\lambda_i=0$ for $i\neq 1$ (which yields a
probability measure by~\cite[Corollary~A.9]{part1}), we see that
$\mathcal{X}$ is also invariant for the SDE with drift
$\wmu(x)=ax+b+\theta(x)\lambda=\wa x+b$, where
$\wa_{11}=a_{11}+4\lambda$, $\wa_{QQ}=a_{QQ}+2\lambda\Id$ and the
remaining coordinates unaltered. Note that
$\wa_{11}\Id-2\wa_{QQ}=a_{11}\Id-2a_{QQ}$, so
conditions~(\ref{eq:aQdiag}) - (\ref{eq:sterker}) are not affected
by such a measure transformation. This gives the result.
\end{proof}
\subsection{Conical state space}
Let $p\geq q>1$. For $x\in\R^p$ we write
$x=(x_1,y,z)\in\R^{1}\times\R^{q-1}\times\R^{p-q}$. We consider
the quadratic form
\[
\Phi(x)=x_1^2-\sum_{i\in Q} x_i^2,
\]
where $Q=\{2,\ldots,q\}$ and define affine matrix-valued functions
$\zeta$ and $\eta$ by
\[
\zeta(x)=\begin{pmatrix}
    x_1 & y^\top \\
    y & x_1\Id
  \end{pmatrix},\quad \eta(x)=\begin{pmatrix}
    0 & 0 & \ldots & 0 \\
    T_{12}(y) & T_{13}(y) & \ldots & T_{q-2,q-1}(y)
  \end{pmatrix},
\]
with $T_{ij}:\R^{q-1}\rightarrow\R^{q-1}$ for $1\leq i<j<q$ given
by $T_{ij}(y)_i=y_j$, $T_{ij}(y)_j=-y_i$, $T_{ij}(y)_k=0$ for
$k\neq i,j$. By applying a reflection, we may assume the state
space $\mathcal{X}$ is of the form
$\mathcal{X}=\{\Phi\geq0\}\cap\{x_1\geq 0\}$. Analogously to
Lemma~\ref{lem:basis} and Lemma~\ref{lem:czeta} we have the
following.
\begin{lemma}\label{lem:basis2}
Consider the linear space
\begin{align}\label{eq:linspace2}
\mathcal{L}=\left\{ a:\R^p\rightarrow\R^q\mbox{ affine }\,|\,
  \begin{pmatrix}
    x_1 & -y^\top
  \end{pmatrix}a(x)=0\mbox{ for all $x$ with }x_1^2=y^\top y
\right\}.
\end{align}
Then a basis for $\mathcal{L}$ is formed by the columns of $\zeta$
and $\eta$.
\end{lemma}
\begin{proof}
Similar to the proof of Lemma~\ref{lem:basis}.
\end{proof}
\begin{lemma}
Consider the linear space
\[
\mathcal{M}=\left\{M:\R^p\rightarrow\R^{q\times q}\mbox{
affine}\,|\, M(x) \mbox{ symmetric and }M^i\in\mathcal{L}\mbox{
for all }x,i \right\}
\]
with $\mathcal{L}$ defined by (\ref{eq:linspace2}). Then a basis
for $\mathcal{M}$ is given by
\[
\mathcal{B}=\{\zeta,\rho(1),\ldots,\rho(q-1)\},
\]
with $\rho(i)$ an affine symmetric-matrix valued function defined
by
\begin{align*}
\rho(i)_{i+1}:x&\mapsto
  \begin{pmatrix}
    x & y^\top
  \end{pmatrix},\\
\rho(i)_{11}:x&\mapsto y_i, &\\
\rho(i)_{jj}:x&\mapsto-y_i,\,\mbox{ for
}j\neq 1,i+1,  \\
 \rho(i)_{jk}:x&\mapsto0,\,\mbox{ if $j,k\neq i$ and $j\neq k$.}
\end{align*}
\end{lemma}
\begin{proof}
Clearly the elements of $\mathcal{B}$ are linearly independent
elements of $\mathcal{M}$. It remains to show that they span
$\mathcal{M}$.

Let $M\in\mathcal{M}$ be arbitrary. By Lemma~\ref{lem:basis2}
there exist matrices $A$ and $B$ such that
\[
M(x)=\zeta(x) A + \eta(x) B.
\]
Write $Q=\{2,\ldots,q\}$, $T(y)=(T_{ij}(y))_{1\leq i<j<q}$ and $B=
  \begin{pmatrix}
    B^1 & \wB
  \end{pmatrix}$. Then the above display reads
\[
M(x)=
  \begin{pmatrix}
    x_1 A_{11} + y^\top A_{Q1} &  x_1 A_{1Q} + y^\top A_{QQ} \\
    y A_{11} + x_1 A_{Q1} +T(y) B^1 & y A_{1Q} + x_1 A_{QQ} +T(y) \wB
  \end{pmatrix}
\]
Symmetry of $M(x)$ yields
\begin{align*}
A_{1Q}&=A_{Q1}^\top,\\
 A_{QQ}&=A_{QQ}^\top,\\ y A_{11}  +T(y) B^1 &=
 A_{QQ}^\top y, \\y A_{1Q}  +T(y)\wB &= (y A_{1Q}
 +T(y)\wB)^\top.
\end{align*}
Since $y^\top T(y)=0$, the second equation together with the third
gives
\[
0=y^\top T(y) B^1= y^\top (A_{QQ}-A_{11}\Id )y,
\]
which implies $A_{QQ}-A_{11}\Id=0$, as $A_{QQ}-A_{11}\Id$ is
symmetric and thus diagonalizable by an orthogonal matrix. Define
\[
N=M-A_{11}\zeta-\sum_{i\in Q} A_{1i}\rho(i).
\]
Then $N\in\mathcal{M}$ and $N$ is of the form
\[
N(x)=
  \begin{pmatrix}
    0 & 0 \\
    0 & \sum_{k\in Q} C^k y_k
  \end{pmatrix},
\]
for some symmetric $((q-1)\times(q-1))$-matrices $C^k$. By
Lemma~\ref{lem:thetais0} it follows that $N=0$.
\end{proof}
\medskip

Unlike the parabolic case, for a general {\em closed} conical
state space we are not able to find a square root
 such that strong existence and
uniqueness for the resulting SDE can be proved.
An exception is the two-dimensional cone, as this is just a
polyhedron which has already been
 covered in Section~\ref{sec:polyhedral}. The following example
 shows that problems appear for closed cones in higher dimensions.
\begin{example}\label{ex:zeta}
For $p=q=3$, a basis for $\mathcal{M}$ is given by
\begin{align*}
\zeta(x)=
  \begin{pmatrix}
    x_1 & y_1 & y_2  \\
    y_1 & x_1 & 0  \\
    y_2 & 0 & x_1
  \end{pmatrix},\quad   \rho(1)(x)&=\begin{pmatrix}
    y_1 & x_1 & 0  \\
    x_1 & y_1 & y_2  \\
    0 & y_2 & -y_1
  \end{pmatrix},\\  \rho(2)(x)&= \begin{pmatrix}
    y_2 & 0 & x_1  \\
    0 & -y_2 & y_1  \\
    x_1 & y_1 & y_2
  \end{pmatrix}.
\end{align*}
Note that not only $\zeta$ but also $\zeta+\rho(1)$ and
$\zeta+\rho(2)$ are positive semi-definite on
$\mathcal{X}=\{\Phi\geq0\}\cap\{x_1\geq0\}=\{x\in\R^3:x_1\geq0,x_1^2\geq
y^\top y\}$. The structure of these matrices appears to be too
complex to compute a manageable square root.
\end{example}
\medskip

However, Proposition~\ref{prop:nablaphi2} enables us to derive
sufficient conditions for stochastic invariance of the {\em open}
conical state space $\{\zeta>0\}$. This can be used to show
existence of a unique strong solution for the affine
SDE~(\ref{eq:SDE}) with square root $\sigma=|\zeta|^{1/2}$, see
the next proposition. Note that this approach is not applicable
for $\zeta+\rho(1)$ and $\zeta+\rho(2)$ in Example~\ref{ex:zeta},
as these matrices are singular on the whole of $\R^3$. We leave
the question of existence of an affine diffusion with a
\emph{closed} conical state space open for further research.
\begin{theorem}\label{th:strongsolconical}
There exists an affine SDE with drift $\mu(x)=a x+b$, diffusion
matrix
\[
\theta(x)=
  \begin{pmatrix}
    x_1 & y^\top \\
    y & x_1\Id
  \end{pmatrix},
\]
and conical state space $\mathcal{X}=\{\theta>0\}=\{x_1>(y^\top
y)^{1/2}\}$ if
\begin{align}
a_{1Q}-a_{Q1}^\top&=0\label{al:admissiblecone1}\\
a_{11}\Id-a_{QQ}&\geq0\label{al:admissiblecone2}\\
b_1-\half p -\|b_Q\|&\geq\label{al:admissiblecone3}0.
\end{align}
\end{theorem}
\begin{proof}
 Let $\sigma=|\theta|^{1/2}$. Then $\sigma$ is locally Lipschitz continuous on
$\mathcal{X}$, so strong existence and uniqueness for
(\ref{eq:SDE}) follows (see~\cite[Theorem~5.2.5]{Karshr}) as soon
as we have shown stochastic invariance of $\mathcal{X}$. It holds
that $\mathcal{X}$ is a connected component of $\{\Phi>0\}$, with
$\Phi(x)=x_1^2-y^\top y$. Therefore, in view of
Proposition~\ref{prop:nablaphi2}, it suffices to prove
(\ref{al:phiv}) and (\ref{al:phiv2}). The first condition is
immediate. For the second one, a calculation shows that
$\nabla\Phi(x)(\mu(x)-\half\sum_{i=1}^p
  (A^i)^i)$ equals
\[
2(a_{11}x_1^2+ x_1(a_{1Q}-a_{Q1}^\top)y + (b_1-\half p) x_1 -
y^\top a_{QQ} y - b_Q^\top y).
\]
This is non-negative for all $x\in\mathcal{X}$ if
(\ref{al:admissiblecone1}) - (\ref{al:admissiblecone3}) hold and
$a_{11}\geq0$. Indeed, in that case we have for
$x\in\mathcal{X}=\{x_1^2> y^\top y \}\cap\{x_1>0\}$ that
\begin{align*}
&a_{11}x_1^2+ x_1(a_{1Q}-a_{Q1}^\top)y + (b_1-\half p) x_1 -
y^\top a_{QQ} y - b_Q^\top y\\&\qquad\geq
y^\top(a_{11}\Id-a_{QQ})y+(b_1-\half p) x_1 -\langle
b_Q,y\rangle\\&\qquad\geq y^\top(a_{11}\Id-a_{QQ})y+(b_1-\half
p-\|b_Q\|) x_1\geq0,
\end{align*}
since $-\langle b_Q,y\rangle \geq -\|b_Q \| \|y\|\geq -\|b_Q \|
x_1$ by Cauchy-Schwarz. The non-negativity of $a_{11}$ can be
dispensed with, by the same arguments as in the proof of
Proposition~\ref{prop:quadrinvar}.
%
\end{proof}
\appendix

\section{Convex analysis}\label{sec:convexgeo}
In this section we state and prove the results on convex analysis
applied in Section~\ref{sec:polyhedral}. Let $\mathcal{X}$ be
given by (\ref{eq:convexpolyX}) and in addition to an affine
function $u$ we are given an affine function $d$ by
\[
d:\R^p\rightarrow\R:x\mapsto a x +b,
\]
for some $a\in\R^{1\times p}$, $b\in\R$.
Proposition~\ref{prop:sht} below is the main result, which yields
Proposition~\ref{prop:driftpos} to tackle the drift and
Proposition~\ref{prop:uiscv} to tackle the diffusion matrix of
affine diffusions with non-canonical polyhedral state space.
\begin{prop}\label{prop:sht}
Suppose $\mathcal{X}\subset\{d\geq0\}$. Then there exist $c\geq0$
and $\lambda\in\R^{1\times q}_{\geq0}$, such that
\[
d=\lambda u+c.
\]
\end{prop}
\begin{proof}
We give a proof by contradiction. Let
\[
\mathcal{K}=\{(\lambda\gamma,\lambda\delta+c):
\lambda\in\R^{1\times q}_{\geq0},c\geq0\}.
\]
Suppose $(a,b)\not\in\mathcal{K}$. Since $\mathcal{K}$ is a closed
convex set, $(a,b)$ is strictly separated from $\mathcal{K}$ by
the Separating Hyperplane Theorem. Therefore, there exist
$y\in\R^p$ and $y_0\in\R$ such that $\langle
(y,y_0),(k,k_0)\rangle>\langle (y,y_0),(a,b)\rangle$ for all
$(k,k_0)\in\mathcal{K}$, i.e.\
\[
k y + k_0 y_0>ay+by_0\quad\mbox{for all $(k,k_0)\in\mathcal{K}$}.
\]
In other words, for all $\lambda_i\geq0$ and $c\geq0$ we have
\[
\sum_i\lambda_i(\gamma_iy+\delta_iy_0) + c y_0>ay+by_0.
\]
It easily follows that
\begin{align}
ay+by_0&<0\\
\gamma_iy+\delta_iy_0&\geq0\\
y_0&\geq0.
\end{align}
Using this we construct $x\in\mathcal{X}$ for which $d(x)<0$.
Suppose $y_0>0$. Then we take $x=y/y_0$. Indeed,
$u_i(x)=(\gamma_iy+\delta_iy_0)/y_0\geq0$, so $x\in\mathcal{X}$.
But $d(x)=(ay+by_0)/y_0<0$, which is a contradiction. Suppose
$y_0=0$. Then we take an arbitrary $x_0\in\mathcal{X}$ and let
$x_N=x_0+Ny$, with $N\in\N$. Then
$u_i(x_N)=u_i(x_0)+N\gamma_iy\geq0$ for all $N$, so
$x_N\in\mathcal{X}$, but $d(x_N)=d(x_0)+Nay<0$ for $N$ big enough.
\end{proof}
\begin{prop}\label{prop:driftpos}
Suppose $\partial{\mathcal{X}}_i\subset\{d\geq 0\}$ for some $i\in
Q$. Then there exist $c\geq0$ and $\lambda\in\R^{1\times q}$ with
$\lambda_{j}\geq0$ for $j\in Q\backslash\{i\}$, such that
\[
d=\lambda u+c.
\]
\end{prop}
\begin{proof}
Let $u_0:=-u_i$. Then $\partial{\mathcal{X}}_i=\bigcap_{j=0}^q
\{u_j\geq0\}$ and $d(x)\geq0$ for $x\in\partial{\mathcal{X}}_i$.
Hence we can apply Proposition~\ref{prop:sht}, which gives the
existence of $\lambda_{j}\geq0$ with $j=0,\ldots,q$ and $c\geq0$
such that
\[
d(x)=\sum_{j=0}^q \lambda_{j} u_j(x)+c=\sum_{j=1}^q
\widetilde{\lambda_{j}} u_j(x)+c,
\]
with $\widetilde{\lambda_{j}}=\lambda_{j}\geq0$ for $j\not=i$ and
$\widetilde{\lambda_{i}}=\lambda_{i}-\lambda_{0}$.
\end{proof}
\begin{lemma}\label{lem:cinietleeg}
Assume $Q$ is minimal. It holds that
$\partial\mathcal{X}_i\not=\emptyset$ for all $i\in Q$.
\end{lemma}
\begin{proof}
Fix $i\leq q$. By minimality of $Q$ we can choose $x\in\R^p$ such
that $u_i(x)<0$ and $u_j(x)\geq0$ for all $j\not=i$. Since
$\mathcal{X}\not=\emptyset$, we can choose $y\in\mathcal{X}$. Then
$u_j(y)\geq 0$ for all $j$. For $t\in[0,1]$ it holds that
\[
u_j(tx+(1-t)y)=t u_j(x)+(1-t)u_j(y),
\]
which is non-negative for $j\not=i$. For
$t=u_i(y)/(u_i(y)-u_i(x))$ we have $u_i(tx+(1-t)y)=0$, so
$tx+(1-t)y\in\partial\mathcal{X}_i$.
\end{proof}
\begin{prop}\label{prop:uiscv}
Assume $Q$ is minimal. Suppose
$\partial{\mathcal{X}}_i\subset\{d=0\}$ for some $i\in Q$. Then
there exists $\lambda_i\in\R$ such that $v(x)=\lambda_i u_i(x)$
for $x\in\mathcal{X}$. If $\mathcal{X}^\circ\not=\emptyset$, then
$v(x)=\lambda_i u_i(x)$ for all $x\in\R^p$.
\end{prop}
\begin{proof}
We have $\partial\mathcal{X}_i\subset\{v\geq0\}$ and
$\partial\mathcal{X}_i\subset\{-v\geq0\}$. Applying
Proposition~\ref{prop:driftpos} with $d=v$ respectively $d=-v$, we
derive that
\begin{align*}
v(x)&=\sum_{j=1}^q \lambda_j u_j(x)+c_1\\
-v(x)&=\sum_{j=1}^q \mu_j u_j(x)+c_2,
\end{align*}
for some $\lambda,\mu\in\R^{1\times p}$ with
$\lambda_j,\mu_j\geq0$ for $j\not=i$ and $c_1,c_2\geq0$. Adding
the equations in the above display gives
\[
0=\sum_{j=1}^q (\lambda_j+\mu_j) u_j(x)+c_1+c_2.
\]
By Lemma \ref{lem:cinietleeg} we can choose
$x\in\partial\mathcal{X}_i$ and deduce that $c_1=c_2=0$. So
\[
-(\lambda_i+\mu_i)u_i(x)=\sum_{j\not=i}(\lambda_j+\mu_j)u_j(x).
\]
By minimality of $Q$ we can choose $x\in\R^p$ such that $u_i(x)<0$
and $u_j(x)\geq0$ for all $j\not=i$. This gives that
$c:=\lambda_i+\mu_i\geq0$. If $c>0$, then for $x\in\mathcal{X}$ we
have
\[
0\leq u_i(x)=-c^{-1}\sum_{j\not=i}(\lambda_j+\mu_j)u_j(x)\leq 0,
\]
whence $u_i(x)=0$ for $x\in\mathcal{X}$. So
$\mathcal{X}=\partial\mathcal{X}_i\subset\{v=0\}$ and
$v(x)=u_i(x)=0$ for $x\in\mathcal{X}$. If $c=0$, then
$\sum_{j\not=i}(\lambda_j+\mu_j)u_j(x)=0$ for all $x$. This holds
in particular for $x\in\mathcal{X}$, i.e.\ for $x$ such that
$u_j(x)\geq0$ for all $j$. Hence for $x\in\mathcal{X}$ we have
$\lambda_ju_j(x)=\mu_ju_j(x)=0$ for all $j\not=i$, so
\begin{align}\label{eq:vonC}
v(x)=\sum_{j=1}^q \lambda_j u_j(x)+c_1=\lambda_i u_i(x),
\end{align}
for $x\in\mathcal{X}$. If $\mathcal{X}^\circ\not=\emptyset$, then
choosing $x\in\mathcal{X}^\circ$ gives $u_j(x)>0$ for all $j$,
which implies $\lambda_j=0$ for all $j\not=i$. Then
(\ref{eq:vonC}) holds for all $x\in\R^p$.
\end{proof}

\bibliographystyle{imsart-nameyear}
\bibliography{refs}

\end{document}